\documentclass[12pt,reqno]{amsart}
\usepackage{xcolor}
\usepackage[pagebackref,colorlinks=true,urlcolor=blue,linkcolor=red,citecolor=blue]{hyperref}
\usepackage[notref,notcite]{}
\usepackage{cancel}

\usepackage{amssymb,verbatim}
\usepackage{amsmath,amsthm}
\usepackage{soul}
\newtheorem{theorem}{Theorem}[section]
\newtheorem{lemma}[theorem]{Lemma}

\newtheorem{corollary}[theorem]{Corollary}

\newtheorem{statement}[theorem]{Statement}
\numberwithin{equation}{section}

\newcommand{\R}{\mathbb{R}}

\newcommand{\rb}{\right)}
\newcommand{\PD}{\partial}

\newcommand{\wt}{\widetilde}

\newcommand{\Sc}{\mathcal{S}}

\newcommand{\Cb}{\mathbb{C}}

\newcommand{\Rb}{\mathbb{R}}
\newcommand{\Sb}{\mathbb{S}}

\newcommand{\lb}{\left(}

\newcommand{\Beq}{\begin{equation}}
\newcommand{\Eeq}{\end{equation}}
\newcommand{\beq}{\begin{equation*}}
\newcommand{\eeq}{\end{equation*}}
\newcommand{\bal}{\begin{align}}
\newcommand{\eal}{\end{align}}

\newcommand{\bp}{\begin{prob}}
	\newcommand{\ep}{\end{prob}}
\newcommand{\bpr}{\begin{proof}}
	\newcommand{\epr}{\end{proof}}

\usepackage{amsmath,tikz}
\newcommand{\vp}{\varphi}

\newcommand*\C{\mathbb{C}}

\def\tbl{\textcolor{blue}}

\def\l{\langle}
\def\r{\rangle}

\title[Range Characterization of Momentum ray transforms]{Momentum Ray Transforms, II: Range Characterization In the Schwartz space}

\author[V.\ P.\ Krishnan, R.\ Manna, S.\ K.\ Sahoo and V.\ A.\ Sharafutdinov]{Venkateswaran P.\ Krishnan$^\ast$$^\dagger$, Ramesh Manna$^\dagger$, Suman Kumar Sahoo$^\dagger$ and Vladimir A.\ Sharafutdinov$^\sharp$}
	
	
	\thanks{The first author was supported by US NSF grant DMS 1616564 and SERB Matrics Grant,
MTR/2017/000837.}
	\thanks{The second author was supported by SERB National Postdoctoral fellowship, PDF/2017/002780.}
		\thanks{The last author was supported by RFBR Grant ??-??-??????.}
	
	\email{vkrishnan@tifrbng.res.in,ramesh@tifrbng.res.in,suman@tifrbng.res.in,\newline\indent sharaf@math.nsc.ru}
	\address{$^\ast$Corresponding author
		\newline\indent$^\dagger$TIFR Centre for Applicable Mathematics, Sharada Nagar, Chikkabommasandra,\newline\indent\hspace{0mm} Yelahanka New Town, Bangalore, India
		\newline\indent$^\sharp$Sobolev Institute of Mathematics; 4 Koptyug Avenue, Novosibirsk, 630090, Russia;
		\newline\indent\hspace{0mm} Novosibirsk State University, 2 Pirogov street, 630090, Russia}

\begin{document}

\begin{abstract}
The momentum ray transform $I^k$ integrates a rank $m$ symmetric tensor field $f$ over lines of ${\R}^n$ with the weight $t^k$:
$
(I^k\!f)(x,\xi)=\int_{-\infty}^\infty t^k\l f(x+t\xi),\xi^m\r\,dt.
$
We give the range characterization for the operator $f\mapsto(I^0\!f,I^1\!f,\dots, I^m\!f)$ on the Schwartz space of rank $m$ smooth fast decaying tensor fields.
In dimensions $n\ge3$, the range is characterized by certain differential equations of order $2(m+1)$ which generalize the classical John equations. In the two-dimensional case, the range is characterized by certain integral conditions which generalize the classical Gelfand -- Helgason -- Ludwig conditions.
\end{abstract}

\maketitle

\section{Introduction}

Starting with the classical paper \cite{J} by F. John, the range characterization for many integral geometry operators is the traditional subject of Integral Geometry.

Let ${\mathcal S}({\R}^n)$ be the Schwartz space of smooth functions rapidly decaying at infinity together with all derivatives (we use the term {\it smooth} as the synonym of {\it $C^\infty$-smooth}). John considers the operator
\begin{equation}
I:{\mathcal S}({\R}^3)\rightarrow C^\infty({\R}^4)
	                                       \label{Eq1.1}
\end{equation}
that integrates a function $f$ over non-horizontal lines
$$
(If)(x_1,x_2,\alpha_1,\alpha_2)=\int\limits_{-\infty}^\infty f(x_1+\alpha_1t,x_2+\alpha_2t,t)\,dt.
$$
The operator \eqref{Eq1.1} and its different generalizations are called {\it ray transforms} (the name {\it X-ray transform} is also widely used).
John proves that a function $\varphi\in C^\infty({\R}^4),\ \varphi=\varphi(x,\alpha)$ belongs to the range of the operator \eqref{Eq1.1} if and only if it fast decays in $x$ and satisfies the second order differential equation
\begin{equation}
\frac{\partial^2\varphi}{\partial x_1\partial\alpha_2}-\frac{\partial^2\varphi}{\partial x_2\partial\alpha_1}=0.
	                                       \label{Eq1.2}
\end{equation}

It is convenient to parameterize the family of oriented lines in ${\mathbb R}^n$ by points of the manifold
$$
T{\mathbb S}^{n-1}=\{(x,\xi)\in{\mathbb R}^n\times{\mathbb R}^n\mid |\xi|=1,\langle x,\xi\rangle=0\}\subset{\mathbb R}^n\times{\mathbb R}^n\tbl{,}
$$
that is, by the tangent bundle of the unit sphere ${\mathbb S}^{n-1}$. Namely, a point $(x,\xi)\in T{\mathbb S}^{n-1}$ determines the line $\{x+t\xi\mid t\in{\mathbb R}\}$. Hereafter $\l\cdot,\cdot\r$ is the standard dot-product on ${\R}^n$ and $|\cdot|$ is the corresponding norm.
Observe that the Schwartz space ${\mathcal S}(E)$ is well defined for a smooth vector bundle $E\rightarrow M$ over a compact manifold $M$. In particular, the space ${\mathcal S}(T{\mathbb S}^{n-1})$ is well defined. The {\it ray transform}
\begin{equation}
I:{\mathcal S}({\R}^n)\rightarrow {\mathcal S}(T{\mathbb S}^{n-1})
	                                       \label{Eq1.3}
\end{equation}
is the linear continuous operator defined by
$$
(If)(x,\xi)=\int\limits_{-\infty}^\infty f(x+t\xi)\,dt.
$$
John's result was generalized to any dimension $n\ge3$ by S. Helgason \cite{Hb}. Instead of \eqref{Eq1.2}, a system of second order differential equations of the same structure appears in the range characterization for the operator \eqref{Eq1.3}.  We do not present the precise statement since it is covered by Theorem \ref{Th1.1} below.

Let $S^{m}\Rb^{n}$ be the complex vector space of rank $m$ symmetric tensors on $\Rb^n$. The dimension of $S^{m}\Rb^{n}$ is ${n+m-1\choose m}$.
In particular, $S^0\Rb^{n}=\Cb$ and $S^1\Rb^{n}={\Cb}^n$.
Let
$\Sc(\Rb^{n}; S^{m}\Rb^{n})$ be the Schwartz space of $S^{m}\Rb^{n}$-valued functions that are called {\it rank $m$ smooth fast decaying symmetric tensor fields} on ${\Rb}^n$. The {\it ray transform} is the linear continuous operator
\begin{equation}
I: \Sc(\Rb^{n}; S^{m}\Rb^{n})\to \Sc (T\Sb^{n-1})
                                   \label{Eq1.4}
\end{equation}
that is defined, for $f=(f_{i_1\dots i_m})\in\Sc(\Rb^{n}; S^{m}\Rb^{n})$, by
\begin{equation}
 (If) (x,\xi)=\int\limits_{-\infty}^\infty f_{i_1\dots i_m}(x+t\xi)\,\xi^{i_1}\dots\xi^{i_m}\,d t=\int\limits_{-\infty}^\infty \l f(x+t\xi),\xi^m\r\,d t\quad\big((x,\xi)\in T{\mathbb S}^{n-1}\big).
                                   \label{Eq1.5}
\end{equation}
We use the Einstein summation rule: the summation from 1 to $n$ is assumed over every index repeated in lower and upper positions in a monomial.
To adopt our formulas to the Einstein summation rule, we use either lower or upper indices for denoting coordinates of vectors and tensors. For instance, $\xi^i=\xi_i$ in \eqref{Eq1.5}. There is no difference between co- and contravariant tensors since we use Cartesian coordinates only.

Observe that the integral \eqref{Eq1.5} makes sense for arbitrary $x\in{\R}^n$ and $0\neq\xi\in{\R}^n$. We define the operator
\begin{equation}
J: \Sc(\Rb^{n}; S^{m}\Rb^{n})\to C^\infty\big({\R}^n\times({\R}^n\setminus\{0\})\big)
                                   \label{Eq1.6}
\end{equation}
by the same formula
\begin{equation}
 (Jf) (x,\xi)=\int\limits_{-\infty}^\infty \l f(x+t\xi),\xi^m\r\,d t\quad\big((x,\xi)\in {\R}^n\times({\R}^n\setminus\{0\})\big).
                                   \label{Eq1.7}
\end{equation}
For a tensor field $f\in{\mathcal S}({\R}^n;S^m{\R}^n)$, the function $\varphi=If$ is recovered from $\psi=Jf$ by $\varphi=\psi|_{T{\mathbb S}^{n-1}}$. On the other hand, $\psi$ can be recovered from $\varphi$. Indeed, as immediately follows from \eqref{Eq1.7}, the function $\psi=Jf$ possesses the following homogeneity in the second argument
$$
\psi(x,t\xi)=\frac{t^m}{|t|}\psi(x,\xi)\quad(0\neq t\in{\R})
$$
and has the following property in the first argument
$$
\psi(x+t\xi,\xi)=\psi(x,\xi)\quad(t\in{\R}).
$$
This implies that
\begin{equation}
\psi(x,\xi)=|\xi|^{m-1}\varphi\Big(x-\frac{\l\xi,x\r}{|\xi|^2}\xi,\frac{\xi}{|\xi|}\Big).
                                         \label{Eq1.8}
\end{equation}
Thus the functions $\varphi=If$ and $\psi=Jf$ give the same information for a tensor field $f$. Therefore the operator \eqref{Eq1.6} is also called the ray transform. The function $\psi=Jf$ is sometimes more convenient since the partial derivatives $\frac{\partial\psi}{\partial x^i}$ and $\frac{\partial\psi}{\partial\xi^i}$ are well defined.

The range characterization for the operator \eqref{Eq1.4} was obtained by V.\ Sharafutdinov. Let us cite Theorem 2.10.1 of \cite{Sh}.

\begin{theorem}  \label{Th1.1}
A function $\varphi\in{\mathcal S}(T{\mathbb S}^{n-1})\ (n\geq3)$ belongs to the range of the operator \eqref{Eq1.4} if and only if the following two conditions hold:
\begin{enumerate}
\item[(1)] $\varphi(x,-\xi)=(-1)^m\varphi(x,\xi)$;

\item[(2)] the function $\psi\in C^\infty\big({\mathbb R}^n\times({\mathbb R}^n\setminus\{0\})\big)$, defined by \eqref{Eq1.8}, satisfies the
equations
\begin{equation}
\Big(\frac{\partial^2}{\partial x^{i_1}\partial\xi^{j_1}}-\frac{\partial^2}{\partial x^{j_1}\partial\xi^{i_1}}\Big)\dots
\Big(\frac{\partial^2}{\partial x^{i_{m+1}}\partial\xi^{j_{m+1}}}-\frac{\partial^2}{\partial x^{j_{m+1}}\partial\xi^{i_{m+1}}}\Big)\psi=0
                                                   \label{Eq1.9}
\end{equation}
for all indices $1\leq i_1,j_1,\dots,i_{m+1},j_{m+1}\leq n$.
\end{enumerate}
\end{theorem}

We call \eqref{Eq1.9} the {\it John equations} and the differential operators
\begin{equation}
J_{ij}=\frac{\partial^2}{\partial x^i\partial\xi^j}-\frac{\partial^2}{\partial x^j\partial\xi^i}:
C^\infty({\R}^n\times{\R}^n)\rightarrow C^\infty({\R}^n\times{\R}^n)
                                                   \label{Eq1.10}
\end{equation}
the {John operators}. In the case of $(m,n)=(0,3)$, Theorem \ref{Th1.1} is equivalent to John's result. In the case of $m=0$, Theorem \ref{Th1.1} was proved by Helgason \cite{Hb}. Just the latter case is used in the current paper. Nevertheless, we have presented the statement in the most generality since it is interesting to compare Theorem \ref{Th1.1} with our main result, Theorem \ref{Th1.3} below.

The system \eqref{Eq1.9} is worth studying by itself. How many linearly independent equations are contained in the system? What is the geometric sense of the system? The present paper does not discuss such questions. See recent papers \cite{D} and \cite{NSV} related to these questions.

Theorem \ref{Th1.1} is definitely false in the case of $n=2$. More precisely, for a tensor field $f\in{\mathcal S}({\R}^2;S^m{\R}^2)$, the John equations \eqref{Eq1.9} are still satisfied by the function $\psi=Jf$; but the John equations are not sufficient for the existence of a tensor field $f\in{\mathcal S}({\R}^2;S^m{\R}^2)$ such that $\psi=Jf$.
Observe that, in the case of $(m,n)=(0,2)$, the operator \eqref{Eq1.4}, which in this case is the same as operator \eqref{Eq1.3} coincides, up to notation,
with the Radon transform on the plane. Unlike \eqref{Eq1.9}, the corresponding consistency conditions for the Radon transform are of integral nature, see
\cite[Chapter 1, Theorem 2.4]{Hb}. These conditions are named the {\it Helgason -- Ludwig conditions}. However, they were, in fact, first written down by I.~Gelfand et. al. \cite[Section 1.6]{G}. The situation is quite similar for tensor fields. Let us cite the result belonging to E.~Pantjukhina \cite{P}.

\begin{theorem} \label{Th1.2}
Let $n\ge2$ and $m\ge0$. If a function $\varphi\in{\mathcal S}(T{\mathbb S}^{n-1})$ belongs to the range of the operator \eqref{Eq1.4}, then

{\rm (1)} $\varphi(x,-\xi)=(-1)^m\varphi(x,\xi)$;

{\rm (2)} for every integer $r\ge0$, there exist homogeneous polynomials $P^r_{i_1\dots i_m}(x)$ of degree $r$ on ${\R}^n$ such that
$$
\int\limits_{\xi^\bot}\varphi(x',\xi)\l x,x'\r^r\,dx'=P^r_{i_1\dots i_m}(x)\xi^{i_1}\dots\xi^{i_m}\quad \big((x,\xi)\in T{\mathbb S}^{n-1}\big),
$$
where $dx'$ is the $(n-1)$-dimensional Lebesgue measure on the hyperplane $\xi^\bot=\{x'\in{\R}^n\mid\l\xi,x'\r=0\}$.

In the case of $n=2$, the converse statement is true: If a function $\varphi\in{\mathcal S}(T{\mathbb S}^1)$ satisfies {\rm (1)} and {\rm (2)} then there exists a tensor field $f\in{\mathcal S}({\R}^2;S^m{\R}^2)$ such that $\varphi=If$.
\end{theorem}

This statement will be used in our proof of Theorem \ref{Th1.4} below.

\bigskip

Now we introduce the subject of the current paper and present our main results.

The {\it momentum ray transforms}
$$
I^k:{\mathcal S}({\R}^n;S^m{\R}^n)\rightarrow{\mathcal S}(T{\mathbb S}^{n-1})
$$
are defined for $k=0,1,\dots$ as follows:
\begin{equation}
(I^k\!f)(x,\xi)=\int\limits_{-\infty}^\infty t^k\l f(x+t\xi),\xi^m\r\,dt\quad\big((x,\xi)\in T{\mathbb S}^{n-1}\big).
                                   \label{Eq1.10a}
\end{equation}
In particular, $I^0=I$. A rank $m$ symmetric tensor field $f$ is uniquely determined by the functions $(I^0\!f,I^1\!f,\dots, I^m\!f)$. The inversion algorithm is presented in \cite{KMSS}.

Quite similarly to \eqref{Eq1.6}, we introduce the operators
$$
J^k: \Sc(\Rb^{n}; S^{m}\Rb^{n})\to C^\infty\big({\R}^n\times({\R}^n\setminus\{0\})\big)
$$
by
\begin{equation}
 (J^kf) (x,\xi)=\int\limits_{-\infty}^\infty t^k\,\l f(x+t\xi),\xi^m\r\, d t\quad\big((x,\xi)\in {\R}^n\times({\R}^n\setminus\{0\})\big).
                                   \label{Eq1.11}
\end{equation}
For a tensor field $f\in{\mathcal S}({\R}^n;S^m{\R}^n)$, the function $\varphi^k=I^kf$ is recovered from $\psi^k=J^kf$ by $\varphi^k=\psi^k|_{T{\mathbb S}^{n-1}}$. On the other hand, $\psi^k$ can be recovered from $(\varphi^0,\dots,\varphi^k)$. Indeed, as immediately follows from \eqref{Eq1.11}, the functions $\psi^k=J^kf$ possess the following homogeneity in the second argument
\begin{equation}
\psi^k(x,t\xi)=\frac{t^{m-k}}{|t|}\psi^k(x,\xi)\quad(0\neq t\in{\R})
                                         \label{Eq1.12}
\end{equation}
and have the following property in the first argument
\begin{equation}
\psi^k(x+t\xi,\xi)=\sum\limits_{\ell=0}^k{k\choose\ell}(-t)^{k-\ell}\psi^\ell(x,\xi)\quad(t\in{\R}).
                                         \label{Eq1.13}
\end{equation}
The two previous formulas imply
\begin{equation}
\psi^k(x,\xi)=|\xi|^{m-2k-1}\sum\limits_{\ell=0}^k(-1)^{k-\ell}{k\choose\ell}|\xi|^\ell\l\xi,x\r^{k-\ell}\,\varphi^\ell
\Big(x-\frac{\l\xi,x\r}{|\xi|^2}\xi,\frac{\xi}{|\xi|}\Big).
                                         \label{Eq1.14}
\end{equation}

Formulas \eqref{Eq1.13} and \eqref{Eq1.14} mean, in particular, that the operator $I^k$ must always be  considered together with lower order momenta $(I^0,\dots,I^{k-1})$, i.e., the data $(I^0f,\dots,I^kf)$ must always be used instead of $I^kf$.

\begin{theorem} \label{Th1.3}
Let $n\ge3$ and $m\ge0$. An $(m+1)$-tuple
$$
(\varphi^0,\dots,\varphi^m)\in\underbrace{{\mathcal S}(T{\mathbb S}^{n-1})\times\dots\times{\mathcal S}(T{\mathbb S}^{n-1})}_{m+1}
$$
belongs to the range of the operator
$$
{\mathcal S}({\R}^n;S^m{\R}^n)\rightarrow\underbrace{{\mathcal S}(T{\mathbb S}^{n-1})\times\dots\times{\mathcal S}(T{\mathbb S}^{n-1})}_{m+1},\quad
f\mapsto(I^0\!f,\dots,I^m\!f)
$$
if and only if the following two conditions are satisfied:\\
{\rm (1)} the functions possess the following evenness in the second argument:
\begin{equation}
\varphi^k(x,-\xi)=(-1)^{m-k}\varphi^k(x,\xi)\quad (0\le k\le m);
                                         \label{Eq1.14a}
\end{equation}
{\rm (2)} 
the function $\psi^m\in C^\infty\big({\R}^n\times({\R}^n\setminus\{0\})\big)$, defined by \eqref{Eq1.14} for $k=m$, satisfies the John conditions
\begin{equation}
J_{i_1j_1}\dots J_{i_{m+1}j_{m+1}}\psi^m=0\quad\mbox{\rm for all}\quad 1\le i_1,j_1,\dots i_{m+1},j_{m+1}\le n,
                                         \label{Eq1.15}
\end{equation}
where the John operators $J_{ij}$ are defined by \eqref{Eq1.10}.
\end{theorem}

Observe that the differential equations \eqref{Eq1.15} are imposed on the function $\psi^m$ only. Nevertheless, all the data $(\varphi^0,\dots,\varphi^m)$ indirectly appear in \eqref{Eq1.15} in view of \eqref{Eq1.14}.

Let us now discuss the two-dimensional case. Fix Cartesian coordinates on ${\R}^2$. For $\xi=(\xi_1,\xi_2)\in{\R}^2$, let $\xi^\bot=(-\xi_2,\xi_1)$.
As before, to adopt our formulas to the Einstein summation rule, we use either lower or upper indices for coordinates of vectors.

\begin{theorem} \label{Th1.4}
Let $m\ge0$. If an $(m+1)$-tuple
$$
(\varphi^0,\dots,\varphi^m)\in\underbrace{{\mathcal S}(T{\mathbb S}^{1})\times\dots\times{\mathcal S}(T{\mathbb S}^{1})}_{m+1}
$$
belongs to the range of the operator,
$$
{\mathcal S}({\R}^2;S^m{\R}^2)\rightarrow\underbrace{{\mathcal S}(T{\mathbb S}^{1})\times\dots\times{\mathcal S}(T{\mathbb S}^{1})}_{m+1},\quad
f\mapsto(I^0\!f,\dots,I^m\!f),
$$
then the following conditions are satisfied.
\begin{enumerate}
\item $\varphi^k(x,-\xi)=(-1)^{m-k}\varphi^k(x,\xi)\ (0\le k\le m)$.

\item For every $r=0,1,2,\dots$ and for every $k=0,1,\dots,m$
$$
\int\limits_{-\infty}^\infty p^r\,\varphi^k(p\xi^\bot,\xi)\,dp=P^{rk}(\xi)
\quad\mbox{for}\quad\xi\in {\mathbb S}^1,
$$
where $P^{rk}(\xi)$ are homogeneous polynomials of degree $r+k+m$ on ${\R}^2$.

\item Polynomials $P^{rk}(\xi)$ are not independent. They are described by the following construction. For
every pair $(\alpha,\beta)$ of non-negative integers there exists a symmetric $m$-tensor
$\mu^{\alpha\beta}=(\mu^{\alpha\beta}_{i_1\dots i_m})\in S^m{\R}^2$ such that
$$
P^{rk}(\xi)=\sum\limits_{\alpha=0}^r\sum\limits_{\beta=0}^k{r\choose\alpha}{k\choose\beta}\mu^{\alpha+\beta,r+k-\alpha-\beta}_{i_1\dots i_m}
(\xi_1^\bot)^\alpha(\xi_2^\bot)^{r-\alpha}\xi_1^{\beta}\xi_2^{k-\beta}\xi^{i_1}\dots\xi^{i_m}.
$$
\end{enumerate}
Conversely, if functions
$\varphi^k\in{\mathcal S}(T{\mathbb S}^{1})\ (k=0,\dots m)$
satisfy conditions {\rm (1)}--{\rm (3)} with some tensors $\mu^{\alpha\beta}\in S^m{\R}^2$, then there exists a tensor field $f\in{\mathcal S}({\R}^2;S^m{\R}^2)$ such that $(\varphi^0,\dots,\varphi^m)=(I^0\!f,\dots,I^m\!f)$.
\end{theorem}

Proofs of Theorems \ref{Th1.3} and \ref{Th1.4} are presented in Sections \ref{S2} and \ref{S3} respectively.

\section{Proof of Theorem \ref{Th1.3}}\label{S2}

\subsection{Preliminaries}
In this section, we prove several preliminary lemmas required for the proof of Theorem \ref{Th1.3}.

\begin{lemma} \label{L2.1}
	If a function $\psi\in C^\infty({\R}^n\setminus\{0\})$ is positively homogeneous of degree $\lambda$, then for any integer $k\ge0$
	$$
	\xi^{i_1}\dots\xi^{i_k}\frac{\partial^k\psi}{\partial\xi^{i_1}\dots\partial\xi^{i_k}}(\xi)=
	\lambda(\lambda-1)\dots(\lambda-k+1)\psi(\xi).
	$$
\end{lemma}

We omit the proof which can be easily done by induction on $k$ on the basis of Euler equation for homogeneous functions.

We will use two first order differential operators:
\[
\l\xi,\partial_x\r=\xi^p\frac{\partial}{\partial x^p},\quad \l\xi,\partial_\xi\r=\xi^p\frac{\partial}{\partial \xi^p}.
\]
The {\it symmetrization} $\sigma(j_1\dots j_k)$ in the indices $(j_1,\dots, j_k)$ is defined by
$$
\sigma(j_1\dots j_k)a_{j_1\dots j_k}=\frac{1}{k!}\sum\limits_{\pi\in\Pi_k}a_{j_{\pi(1)}\dots j_{\pi(k)}},
$$
where $\Pi_k$ is the set of all permutations of the set $\{1,\dots,k\}$.

\begin{lemma} \label{L2.2}
	For any non-negative integers $k$ and $\ell$ and for all indices $(j_1,\dots,j_k)$ satisfying $1\le j_1,\dots,j_k\le n$, the commutator formula
	\begin{equation}
	\l\xi,\partial_x\r^\ell\frac{\partial^k}{\partial\xi^{j_1}\dots\partial\xi^{j_k}}
	=\sigma(j_1\dots j_k)\sum\limits_{p=0}^k(-1)^p{k\choose p}\frac{\ell!}{(\ell\!-\!p)!}\,
	\frac{\partial^k}{\partial x^{j_1}\dots\partial x^{j_p}\partial\xi^{j_{p+1}}\dots\partial\xi^{j_k}}\l\xi,\partial_x\r^{\ell-p}
	                              \label{Eq2.1}
	\end{equation}
	holds under the agreement: $\l\xi,\partial_x\r^r=0$ for $r<0$.
\end{lemma}

\begin{proof}
	We prove the statement by induction on $k$. The statement trivially holds for $k=0$. Assuming \eqref{Eq2.1} to be valid for some $k$, we apply  $\frac{\partial}{\partial\xi^{j_{k+1}}}$ to \eqref{Eq2.1} from the left
	$$
	\begin{aligned}
	&\l\xi,\partial_x\r^\ell\,\frac{\partial^{k+1}}{\partial\xi^{j_1}\dots\partial\xi^{j_{k+1}}}
	+\ell\,\l\xi,\partial_x\r^{\ell-1}\,\frac{\partial^{k+1}}{\partial x^{j_{k+1}}\partial\xi^{j_1}\dots\partial\xi^{j_{k}}}=\\
	&=\sigma(j_1\dots j_k)\sum\limits_{p=0}^k(-1)^p{k\choose p}\frac{\ell!}{(\ell -p)!}\,
	\frac{\partial^{k+1}}{\partial x^{j_1}\dots\partial x^{j_p}\partial\xi^{j_{p+1}}\dots\partial\xi^{j_{k+1}}}\l\xi,\partial_x\r^{\ell-p}.
	\end{aligned}
	$$
	We write this in the form
	\begin{equation}
	\begin{aligned}
	\l\xi,\partial_x\r^\ell\,\frac{\partial^{k+1}}{\partial\xi^{j_1}\dots\partial\xi^{j_{k+1}}}
	&=\sigma(j_1\dots j_k)\sum\limits_{p=0}^k\frac{(-1)^p\ell!}{(\ell -p)!}{k\choose p}\,
	\frac{\partial^{k+1}}{\partial x^{j_1}\dots\partial x^{j_p}\partial\xi^{j_{p+1}}\dots\partial\xi^{j_{k+1}}}\l\xi,\partial_x\r^{\ell-p}\\
	&-\ell\,\l\xi,\partial_x\r^{\ell-1}\,\frac{\partial^{k+1}}{\partial x^{j_{k+1}}\partial\xi^{j_1}\dots\partial\xi^{j_{k}}}.
	\end{aligned}
	                                            \label{Eq2.2}
	\end{equation}
		By the same induction hypothesis \eqref{Eq2.1},
	$$
\begin{aligned}
	\l\xi,\partial_x\r^{\ell-1}\,&\frac{\partial^{k}}{\partial\xi^{j_1}\dots\partial\xi^{j_{k}}}=\\
	&=\sigma(j_1\dots j_k)\sum\limits_{p=0}^k(-1)^p{k\choose p}\frac{(\ell-1)!}{(\ell -p-1)!}\,
	\frac{\partial^k}{\partial x^{j_1}\dots\partial x^{j_p}\partial\xi^{j_{p+1}}\dots\partial\xi^{j_k}}\l\xi,\partial_x\r^{\ell-p-1}.
\end{aligned}
	$$
	We apply $\frac{\partial}{\partial x^{j_{k+1}}}$ from the left to the above equality. Since the latter derivative commutes with $\l\xi,\partial_x\r$, the result can be written as
	$$
	\begin{aligned}
	&\l\xi,\partial_x\r^{\ell-1}\,\frac{\partial^{k+1}}{\partial x^{j_{k+1}}\partial\xi^{j_1}\dots\partial\xi^{j_{k}}}=\\
	&=\sigma(j_1\dots j_k)\sum\limits_{p=0}^k(-1)^p{k\choose p}\frac{(\ell-1)!}{(\ell -p-1)!}\,
	\frac{\partial^{k+1}}{\partial x^{j_1}\dots\partial x^{j_p}\partial x^{j_{k+1}}\partial\xi^{j_{p+1}}\dots\partial\xi^{j_k}}\l\xi,\partial_x\r^{\ell-p-1}.
	\end{aligned}
	$$
	Substitute this expression into \eqref{Eq2.2}
	$$
	\begin{aligned}
	&\l\xi,\partial_x\r^\ell\,\frac{\partial^{k+1}}{\partial\xi^{j_1}\dots\partial\xi^{j_{k+1}}}
	=\sigma(j_1\dots j_k)\sum\limits_{p=0}^k\frac{(-1)^p\ell!}{(\ell -p)!}{k\choose p}
	\frac{\partial^{k+1}}{\partial x^{j_1}\dots\partial x^{j_p}\partial\xi^{j_{p+1}}\dots\partial\xi^{j_{k+1}}}\l\xi,\partial_x\r^{\ell-p}\\
	&-\ell\,\sigma(j_1\dots j_k)\sum\limits_{p=0}^k(-1)^p{k\choose p}\frac{(\ell-1)!}{(\ell -p-1)!}\,
	\frac{\partial^{k+1}}{\partial x^{j_1}\dots\partial x^{j_p}\partial x^{j_{k+1}}\partial\xi^{j_{p+1}}\dots\partial\xi^{j_k}}\l\xi,\partial_x\r^{\ell-p-1}.
	\end{aligned}
	$$
	The symmetrization $\sigma(j_1\dots j_k)$ can be replaced with $\sigma(j_1\dots j_{k+1})$ since the left-hand side is symmetric in indices
	$(j_1,\dots, j_{k+1})$. After the replacement, we can write these indices in the lexicographic order. In this way the formula becomes
	$$
	\begin{aligned}
	&\l\xi,\partial_x\r^\ell\,\frac{\partial^{k+1}}{\partial\xi^{j_1}\dots\partial\xi^{j_{k+1}}}
	=\sigma(j_1\dots j_{k+1})\sum\limits_{p=0}^k\frac{(-1)^p\ell!}{(\ell -p)!}{k\choose p}
	\frac{\partial^{k+1}}{\partial x^{j_1}\dots\partial x^{j_p}\partial\xi^{j_{p+1}}\dots\partial\xi^{j_{k+1}}}\l\xi,\partial_x\r^{\ell-p}\\
	&+\sigma(j_1\dots j_{k+1})\sum\limits_{p=0}^k(-1)^{p+1}{k\choose p}\frac{\ell!}{(\ell -p-1)!}\,\,
	\frac{\partial^{k+1}}{\partial x^{j_1}\dots\partial x^{j_{p+1}}\partial\xi^{j_{p+2}}\dots\partial\xi^{j_{k+1}}}\l\xi,\partial_x\r^{\ell-p-1}.
	\end{aligned}
	$$
	In the second sum, we change the summation index as $p:=p-1$
	$$
	\begin{aligned}
	&\l\xi,\partial_x\r^\ell\,\frac{\partial^{k+1}}{\partial\xi^{j_1}\dots\partial\xi^{j_{k+1}}}
	=\sigma(j_1\dots j_{k+1})\sum\limits_{p=0}^k\frac{(-1)^p\ell!}{(\ell -p)!}{k\choose p}
	\frac{\partial^{k+1}}{\partial x^{j_1}\dots\partial x^{j_p}\partial\xi^{j_{p+1}}\dots\partial\xi^{j_{k+1}}}\l\xi,\partial_x\r^{\ell-p}\\
	&+\sigma(j_1\dots j_{k+1})\sum\limits_{p=1}^{k+1}(-1)^p{k\choose {p-1}}\frac{\ell!}{(\ell -p)!}\,
	\frac{\partial^{k+1}}{\partial x^{j_1}\dots\partial x^{j_p}\partial\xi^{j_{p+1}}\dots\partial\xi^{j_{k+1}}}\l\xi,\partial_x\r^{\ell-p}.
	\end{aligned}
	$$
	Under the agreement ${r\choose s}=0$ if either $s<0$ or $r<s$, both summations can be extended to the limits $0\le p\le k+1$, and the formula becomes
	$$
	\begin{aligned}
	&\l\xi,\partial_x\r^\ell\,\frac{\partial^{k+1}}{\partial\xi^{j_1}\dots\partial\xi^{j_{k+1}}}=\\
	&=\sigma(j_1\dots j_{k+1})\sum\limits_{p=0}^{k+1}(-1)^p\Big[{k\choose p}+{k\choose {p\!-\!1}}\Big]\frac{\ell!}{(\ell -p)!}\,
	\frac{\partial^{k+1}}{\partial x^{j_1}\dots\partial x^{j_p}\partial\xi^{j_{p+1}}\dots\partial\xi^{j_{k+1}}}\l\xi,\partial_x\r^{\ell-p}.
	\end{aligned}
	$$
	With the help of the Pascal relation ${k\choose p}+{k\choose {p-1}}={{k+1}\choose p}$, the formula takes the final form
	$$
	\begin{aligned}
	\l\xi,\partial_x\r^\ell\,&\frac{\partial^{k+1}}{\partial\xi^{j_1}\dots\partial\xi^{j_{k+1}}}=\\
	&=\sigma(j_1\dots j_{k+1})\sum\limits_{p=0}^{k+1}(-1)^p{{k+1}\choose p}\frac{\ell!}{(\ell -p)!}\,
	\frac{\partial^{k+1}}{\partial x^{j_1}\dots\partial x^{j_p}\partial\xi^{j_{p+1}}\dots\partial\xi^{j_{k+1}}}\l\xi,\partial_x\r^{\ell-p}.
	\end{aligned}
	$$
	This finishes the induction step.
\end{proof}

\begin{corollary} \label{C2.1}
For integers $0\le k\le m$, the identity
\begin{equation}
\begin{aligned}
\sigma(j_1\dots j_m)\sum\limits_{k=0}^m\frac{1}{(m-k)!}\l\xi,\partial_x\r
&\frac{\partial^m}{\partial x^{j_1}\dots\partial x^{j_k}\partial \xi^{j_{k+1}}\dots\partial\xi^{j_m}}\l\xi,\partial_x\r^{m-k}=\\
&=\frac{1}{m!}\frac{\partial^m}{\partial\xi^{j_1}\dots\partial\xi^{j_m}}\l\xi,\partial_x\r^{m+1}
\end{aligned}
                       	\label{Eq2.2a}
\end{equation}
holds for all $1\le j_1,\dots,j_m\le n$.
\end{corollary}

\begin{proof}
Denote the left-hand side of \eqref{Eq2.2a} by $A$. Indices can be written in an arbitrary order because of the presence of the symmetrization $\sigma(j_1\dots j_m)$. In particular.
$$
A=\sigma(j_1\dots j_m)\sum\limits_{k=0}^m\frac{1}{(m-k)!}\l\xi,\partial_x\r
\frac{\partial^m}{\partial x^{j_{m-k+1}}\dots\partial x^{j_m}\partial\xi^{j_{1}}\dots\partial\xi^{j_{m-k}}}\l\xi,\partial_x\r^{m-k}.
$$
Since $\l\xi,\partial_x\r$ commutes with derivatives $\frac{\partial}{\partial x^j}$, this can be written as
\begin{equation}
A=\sigma(j_1\dots j_m)\sum\limits_{k=0}^m\frac{1}{(m-k)!}
\frac{\partial^k}{\partial x^{j_{m-k+1}}\dots\partial x^{j_m}}
\l\xi,\partial_x\r
\frac{\partial^{m-k}}{\partial\xi^{j_{1}}\dots\partial\xi^{j_{m-k}}}\l\xi,\partial_x\r^{m-k}.
                       	\label{Eq2.2b}
\end{equation}
By Lemma \ref{L2.2},
$$
\l\xi,\partial_x\r\frac{\partial^{m-k}}{\partial\xi^{j_{1}}\dots\partial\xi^{j_{m-k}}}
=\frac{\partial^{m-k}}{\partial\xi^{j_{1}}\dots\partial\xi^{j_{m-k}}}\l\xi,\partial_x\r
-(m-k)\sigma(j_1\dots j_{m-k})\frac{\partial^{m-k}}{\partial x^{j_{1}}\partial\xi^{j_{2}}\dots\partial\xi^{j_{m-k}}}.
$$
While substituting this expression into \eqref{Eq2.2b}, we can omit the symmetrization $\sigma(j_1\dots j_{m-k})$ because of the presence of the stronger symmetrization $\sigma(j_1\dots j_{m})$. In this way we obtain
$$
\begin{aligned}
A&=\sigma(j_1\dots j_m)\sum\limits_{k=0}^m\frac{1}{(m-k)!}
\frac{\partial^m}{\partial x^{j_{m-k+1}}\dots\partial x^{j_m}\partial\xi^{j_{1}}\dots\partial\xi^{j_{m-k}}}\l\xi,\partial_x\r^{m-k+1}\\
&-\sigma(j_1\dots j_m)\sum\limits_{k=0}^m\frac{m-k}{(m-k)!}
\frac{\partial^m}{\partial x^{j_{m-k+1}}\dots\partial x^{j_m}\partial x^{j_{1}}\partial\xi^{j_{2}}\dots\partial\xi^{j_{m-k}}}\l\xi,\partial_x\r^{m-k}.
\end{aligned}
$$
In the second sum, we can reduce summation limits to $0\le k\le m-1$ because of the presence of the factor $m-k$. Besides this, we can again write indices in the lexicographic order. The formula becomes
$$
\begin{aligned}
A&=\sigma(j_1\dots j_m)\sum\limits_{k=0}^m\frac{1}{(m-k)!}
\frac{\partial^m}{\partial x^{j_1}\dots\partial x^{j_k}\partial\xi^{j_{k+1}}\dots\partial\xi^{j_{m}}}\l\xi,\partial_x\r^{m-k+1}\\
&-\sigma(j_1\dots j_m)\sum\limits_{k=0}^{m-1}\frac{1}{(m-k-1)!}
\frac{\partial^m}{\partial x^{j_{1}}\dots\partial x^{j_{k+1}}\partial\xi^{j_{k+2}}\dots\partial\xi^{j_{m}}}\l\xi,\partial_x\r^{m-k}.
\end{aligned}
$$
We distinguish the summand of the first sum corresponding to $k=0$ and change the summation variable in the second sum as $k:=k-1$. The formula takes the form
$$
\begin{aligned}
A&=\frac{1}{m!}\frac{\partial^m}{\partial\xi^{j_{1}}\dots\partial\xi^{j_{m}}}\l\xi,\partial_x\r^{m+1}\\
&+\sigma(j_1\dots j_m)\sum\limits_{k=1}^m\frac{1}{(m-k)!}
\frac{\partial^m}{\partial x^{j_1}\dots\partial x^{j_k}\partial\xi^{j_{k+1}}\dots\partial\xi^{j_{m}}}\l\xi,\partial_x\r^{m-k+1}\\
&-\sigma(j_1\dots j_m)\sum\limits_{k=1}^{m}\frac{1}{(m-k)!}
\frac{\partial^m}{\partial x^{j_{1}}\dots\partial x^{j_{k}}\partial\xi^{j_{k+1}}\dots\partial\xi^{j_{m}}}\l\xi,\partial_x\r^{m-k+1}.
\end{aligned}
$$
Two sums on the right-hand side cancel each other and we arrive at \eqref{Eq2.2a}.
\end{proof}

\begin{lemma} \label{L2.3}
	For a function $f\in{\mathcal S}({\R}^n)$, the equalities
	\begin{equation}
	\frac{\partial^\ell(J^kf)}{\partial x^{j_1}\dots\partial x^{j_\ell}}=
	\frac{\partial^\ell(J^{k-\ell}f)}{\partial \xi^{j_1}\dots\partial \xi^{j_\ell}}
	                          \label{Eq2.3}
	\end{equation}
	hold for all $0\le\ell\le k$ and for all indices $1\le j_1,\dots,j_\ell\le n$.
\end{lemma}

\begin{proof}
	Applying the operator $\frac{\partial^\ell}{\partial x^{j_1}\dots\partial x^{j_\ell}}$ to the equality
	$$
	(J^k\!f)(x,\xi)=\int\limits_{-\infty}^\infty t^k\, f(x+t\xi)\,dt,
	$$
	we obtain
	\begin{equation}
	\frac{\partial^\ell(J^kf)}{\partial x^{j_1}\dots\partial x^{j_\ell}}(x,\xi)=
	\int\limits_{-\infty}^\infty t^k\,
	\frac{\partial^\ell f}{\partial x^{j_1}\dots\partial x^{j_\ell}}(x+t\xi)\,dt.
	\label{Eq2.4}
	\end{equation}
	On the other hand, applying the operator $\frac{\partial^\ell}{\partial \xi^{j_1}\dots\partial \xi^{j_\ell}}$ to the equality
	$$
	(J^{k-\ell}\!f)(x,\xi)=\int\limits_{-\infty}^\infty t^{k-\ell}\, f(x+t\xi)\,dt,
	$$
	we obtain
	\begin{equation}
	\frac{\partial^\ell(J^{k-\ell}f)}{\partial \xi^{j_1}\dots\partial \xi^{j_\ell}}(x,\xi)=
	\int\limits_{-\infty}^\infty t^k\,
	\frac{\partial^\ell f}{\partial x^{j_1}\dots\partial x^{j_\ell}}(x+t\xi)\,dt.
	                          \label{Eq2.5}
	\end{equation}
	Formulas \eqref{Eq2.4} and \eqref{Eq2.5} imply \eqref{Eq2.3}.
\end{proof}

\begin{lemma} \label{L2.4}
	Given an integer $m\ge0$, define the operators
	$$
	J^{m,k}:{\mathcal S}({\R}^n;S^k{\R}^n)\rightarrow C^\infty\big({\R}^n\times({\R}^n\setminus\{0\})\big)\quad(k=0,1,\dots)
	$$
	by
	\begin{equation}
	(J^{m,k}\!f)(x,\xi)=\int\limits_{-\infty}^\infty t^m\,\l f(x+t\xi),\xi^k\r\,dt.
	\label{Eq2.6}
	\end{equation}
	Then for every $k$ and every $f\in{\mathcal S}({\R}^n;S^k{\R}^n)$,
	\begin{equation}
	J_{i_1j_1}\dots J_{i_{k+1}j_{k+1}}(J^{m,k}\!f)=0\quad\mbox{for all}\quad 1\le i_1,j_1,\dots i_{k+1},j_{k+1}\le n,
	\label{Eq2.7}
	\end{equation}
	where $J_{ij}=\frac{\partial^2}{\partial x^i\partial\xi^j}-\frac{\partial^2}{\partial x^j\partial\xi^i}$ is the John operator.
\end{lemma}

{\bf Remark.} We hope the reader is not confused by using the letter $J$ in two different senses: $J_{ij}$ is the John operator while $J^{m,k}$ is defined by \eqref{Eq2.6}. Both notations are standard.

\begin{proof}
In the case of $k=0$, \eqref{Eq2.6} reads
$$
(J^{m,0}\!f)(x,\xi)=\int\limits_{-\infty}^\infty t^m\, f(x+t\xi)\,dt.
$$
Differentiating this equality, we obtain
$$
\frac{\partial^2(J^{m,0}\!f)}{\partial x^i\partial\xi^j}(x,\xi)=\int\limits_{-\infty}^\infty\! t^{m+1}
\frac{\partial^2 f}{\partial x^i\partial x^j}(x+t\xi)\,dt.
$$
The right-hand side is symmetric in $(i,j)$. Therefore
$$
J_{ij}(J^{m,0}\!f)=
\Big(\frac{\partial^2(J^{m,0}\!f)}{\partial x^i\partial\xi^j}-\frac{\partial^2(J^{m,0}\!f)}{\partial x^j\partial\xi^i}\Big)=0.
$$
This proves the statement of the lemma for $k=0$. Then we proceed by induction in $k$.

Let $k>0$. Differentiating \eqref{Eq2.6}, we obtain
$$
\frac{\partial^2(J^{m,k}\!f)}{\partial x^i\partial\xi^j}(x,\xi)=\int\limits_{-\infty}^\infty\! t^{m+1}
\Big\langle \frac{\partial^2 f}{\partial x^i\partial x^j}(x+t\xi),\xi^k\Big\rangle\,dt
+k\!\int\limits_{-\infty}^\infty\! t^m
\frac{\partial f_{i_1\dots i_{k-1}j}}{\partial x^i}(x+t\xi)\,\xi^{i_1}\dots\xi^{i_{k-1}}\,dt.
$$
The first summand on the right-hand side of this equality is symmetric in $(i,j)$. Therefore
\begin{equation}
\big(J_{ij}(J^{m,k}\!f)\big)(x,\xi)=k\int\limits_{-\infty}^\infty t^m\,
\Big(\frac{\partial f_{i_1\dots i_{k-1}j}}{\partial x^i}(x+t\xi)
-\frac{\partial f_{i_1\dots i_{k-1}i}}{\partial x^j}(x+t\xi)\Big)\xi^{i_1}\dots\xi^{i_{k-1}}\,dt.
                       \label{Eq2.8}
\end{equation}

Let us fix indices $i$ and $j$ and define the tensor field
$h_{ij}\in {\mathcal S}({\R}^n;S^{k-1}{\R}^n)$ by
$$
(h_{ij})_{i_1\dots i_{k-1}}=k\Big(\frac{\partial f_{i_1\dots i_{k-1}j}}{\partial x^i}-\frac{\partial f_{i_1\dots i_{k-1}i}}{\partial x^j}\Big).
$$
Then \eqref{Eq2.8} can be written as
\begin{equation}
J_{ij}(J^{m,k}\!f)=J^{m,k-1}h_{ij}.
                         \label{Eq2.9}
\end{equation}
By the induction hypothesis,
$$
J_{i_1j_1}\dots J_{i_mj_m}(J^{m,k-1}h_{ij})=0\quad\mbox{for all}\quad 1\le i_1,j_1,\dots,i_m,j_m\le n.
$$
Together with \eqref{Eq2.9}, this gives
$$
J_{i_1j_1}\dots J_{i_mj_m}J_{ij}(J^{m,k}\!f)=0\quad\mbox{for all}\quad 1\le i_1,j_1,\dots,i_m,j_m,i,j\le n.
$$
This coincides with \eqref{Eq2.7}.
\end{proof}

\begin{lemma} \label{L2.5}
Let a function $\psi\in C^\infty\big({\R}^n\times{\R}^n\setminus\{0\})\big)$ be positively homogeneous in the second argument
\begin{equation}
\psi(x,t\xi)=t^\lambda\psi(x,\xi)\quad(t>0).
                                         \label{Eq2.10}
\end{equation}
Assume the restriction $\psi|_{T{\mathbb S}^{n-1}}$ to belong to ${\mathcal S}(T{\mathbb S}^{n-1})$. Assume also that restrictions to ${\mathcal S}(T{\mathbb S}^{n-1})$ of the function $\l\xi,\partial_x\r\psi$ and of all its derivatives belong to ${\mathcal S}(T{\mathbb S}^{n-1})$, i.e.,
\begin{equation}
\left.\frac{\partial^{k+\ell}(\l\xi,\partial_x\r\psi)}{\partial x^{i_1}\dots\partial x^{i_k}\partial \xi^{j_1}\dots\partial \xi^{j_\ell}}\right|_{T{\mathbb S}^{n-1}}\in{\mathcal S}(T{\mathbb S}^{n-1})\quad\mbox{for all}\quad 1\le i_1,\dots,i_k,j_1,\dots,j_\ell\le n.
                                         \label{Eq2.11}
\end{equation}
Then the restriction to ${\mathcal S}(T{\mathbb S}^{n-1})$ of every derivative of $\psi$ also belongs to ${\mathcal S}(T{\mathbb S}^{n-1})$, i.e.,
\begin{equation}
\left.\frac{\partial^{k+\ell}\psi}{\partial x^{i_1}\dots\partial x^{i_k}\partial \xi^{j_1}\dots\partial \xi^{j_\ell}}\right|_{T{\mathbb S}^{n-1}}
\in{\mathcal S}(T{\mathbb S}^{n-1})\quad\mbox{for all}\quad 1\le i_1,\dots,i_k,j_1,\dots,j_\ell\le n.
                                         \label{Eq2.12}
\end{equation}
\end{lemma}

\begin{proof}
For $1\le i\le n$, we define first order differential operators on ${\R}^n\times({\R}^n\setminus\{0\})$
\begin{equation}
{\tilde X}_i=\frac{\partial}{\partial x^i}-\xi_i\l\xi,\partial_x\r,\quad
{\tilde \Xi}_i=\frac{\partial}{\partial \xi^i}-x_i\l\xi,\partial_x\r-\xi_i\l\xi,\partial_\xi\r.
                                         \label{Eq2.13}
\end{equation}
Being considered as vector fields on ${\R}^n\times({\R}^n\setminus\{0\})$, ${\tilde X}_i$ and ${\tilde \Xi}_i$ are tangent to $T{\mathbb S}^{n-1}$ at every point of the latter submanifold, see \cite[Lemma 4.1]{KMSS}. Let $X_i$ and $\Xi_i$ be the restrictions to $T{\mathbb S}^{n-1}$ of ${\tilde X}_i$ and ${\tilde \Xi}_i$ respectively. Thus, $X_i$ and $\Xi_i$ are well defined first order differential operators on $T{\mathbb S}^{n-1}$. Obviously,
$$
X_i,\Xi_i:{\mathcal S}(T{\mathbb S}^{n-1})\rightarrow{\mathcal S}(T{\mathbb S}^{n-1})\quad(1\le i\le n).
$$

Let $\varphi=\psi|_{T{\mathbb S}^{n-1}}\in{\mathcal S}(T{\mathbb S}^{n-1})$. Then
\begin{equation}
{\tilde X}_i\psi|_{T{\mathbb S}^{n-1}}=X_i\varphi,\quad {\tilde \Xi}_i\psi|_{T{\mathbb S}^{n-1}}=\Xi_i\varphi.
                                         \label{Eq2.14}
\end{equation}

By the Euler equation for homogeneous functions,
$$
\l\xi,\partial_\xi\r\psi=\lambda\psi.
$$
With the help of the last formula, \eqref{Eq2.13} gives
$$
\frac{\partial\psi}{\partial x^i}={\tilde X}_i\psi+\xi_i\l\xi,\partial_x\r\psi,\quad
\frac{\partial\psi}{\partial \xi^i}={\tilde \Xi}_i\psi+x_i\l\xi,\partial_x\r\psi+\lambda\xi_i\psi.
$$
Together with \eqref{Eq2.14}, this implies
$$
\begin{aligned}
\left.\frac{\partial\psi}{\partial x^i}\right|_{T{\mathbb S}^{n-1}}&=X_i\varphi
+(\xi_i\l\xi,\partial_x\r\psi)|_{T{\mathbb S}^{n-1}},\\
\left.\frac{\partial\psi}{\partial \xi^i}\right|_{T{\mathbb S}^{n-1}}&=\Xi_i\varphi
+(x_i\l\xi,\partial_x\r\psi)|_{T{\mathbb S}^{n-1}}+\lambda\xi_i\varphi.
\end{aligned}
$$
All terms on right-hand sides of these equalities belong to ${\mathcal S}(T{\mathbb S}^{n-1})$. Thus,
$$
\left.\frac{\partial\psi}{\partial x^i}\right|_{T{\mathbb S}^{n-1}}\in{\mathcal S}(T{\mathbb S}^{n-1}),\quad
\left.\frac{\partial\psi}{\partial \xi^i}\right|_{T{\mathbb S}^{n-1}}\in{\mathcal S}(T{\mathbb S}^{n-1}).
$$
We have thus proved the statement of the lemma for first order derivatives. In the case of $\ell=0$, \eqref{Eq2.12} is proved actually in the same way.
Indeed, since the operator $\l\xi,\partial_x\r$ commutes with the derivatives $\frac{\partial}{\partial x^i}$, every derivative $\frac{\partial^k\psi}{\partial x^{i_i}\dots\partial x^{i_k}}$ satisfy hypotheses \eqref{Eq2.10} and \eqref{Eq2.11} of the lemma. On using this fact, we easily prove \eqref{Eq2.12} for $\ell=0$ by induction in $k$.

In the case of $\ell\neq0$, the proof is more complicated since the operator $\l\xi,\partial_x\r$ does not commute with the derivatives $\frac{\partial}{\partial \xi^j}$. Nevertheless, we have the commutator formula
\begin{equation}
\l\xi,\partial_x\r\frac{\partial^\ell}{\partial \xi^{j_1}\dots\partial \xi^{j_\ell}}
=\frac{\partial^\ell}{\partial \xi^{j_1}\dots\partial \xi^{j_\ell}}\l\xi,\partial_x\r
-\ell\,\sigma(j_1\dots j_\ell)\frac{\partial^\ell}{\partial x^{j_1}\partial \xi^{j_2}\dots\partial \xi^{j_\ell}}
                                         \label{Eq2.15}
\end{equation}
which is a partial case of Lemma \ref{L2.2}.

Now we prove \eqref{Eq2.12} by induction on $\ell$. Assume that
$$
\left.\frac{\partial^{k+s}\psi}{\partial x^{i_1}\dots\partial x^{i_k}\partial \xi^{j_1}\dots\partial \xi^{j_s}}\right|_{T{\mathbb S}^{n-1}}
\in {\mathcal S}(T{\mathbb S}^{n-1})\quad\mbox{for}\quad s\le\ell
$$
with some $\ell$.

Let $\ell\ge1$. From \eqref{Eq2.13},
$$
\begin{aligned}
\frac{\partial^{k+\ell+1}\psi}{\partial x^{i_1}\dots\partial x^{i_k}\partial \xi^{j_1}\dots\partial \xi^{j_{\ell+1}}}
&=\tilde\Xi_{j_{\ell+1}}\frac{\partial^{k+\ell}\psi}{\partial x^{i_1}\dots\partial x^{i_k}\partial \xi^{j_1}\dots\partial \xi^{j_\ell}}\\
&+x_{j_{\ell+1}}\l\xi,\partial_x\r\frac{\partial^{k+\ell}\psi}{\partial x^{i_1}\dots\partial x^{i_k}\partial \xi^{j_1}\dots\partial \xi^{j_\ell}}\\
&+\xi_{j_{\ell+1}}\l\xi,\partial_\xi\r\frac{\partial^{k+\ell}\psi}{\partial x^{i_1}\dots\partial x^{i_k}\partial \xi^{j_1}\dots\partial \xi^{j_\ell}}.
\end{aligned}
$$
Since $\frac{\partial^{k+\ell}\psi}{\partial x^{i_1}\dots\partial x^{i_k}\partial \xi^{j_1}\dots\partial \xi^{j_\ell}}$ is positively homogeneous of degree $\lambda-\ell$ in $\xi$, the formula becomes
\begin{equation}
\begin{aligned}
\frac{\partial^{k+\ell+1}\psi}{\partial x^{i_1}\dots\partial x^{i_k}\partial \xi^{j_1}\dots\partial \xi^{j_{\ell+1}}}
&=\tilde\Xi_{j_{\ell+1}}\frac{\partial^{k+\ell}\psi}{\partial x^{i_1}\dots\partial x^{i_k}\partial \xi^{j_1}\dots\partial \xi^{j_\ell}}\\
&+x_{j_{\ell+1}}\l\xi,\partial_x\r\frac{\partial^{k+\ell}\psi}{\partial x^{i_1}\dots\partial x^{i_k}\partial \xi^{j_1}\dots\partial \xi^{j_\ell}}\\
&+(\lambda-\ell)\xi_{j_{\ell+1}}\frac{\partial^{k+\ell}\psi}{\partial x^{i_1}\dots\partial x^{i_k}\partial \xi^{j_1}\dots\partial \xi^{j_\ell}}.
\end{aligned}
                                         \label{Eq2.16}
\end{equation}

We transform the second term on the right hand side of \eqref{Eq2.16} with the help of \eqref{Eq2.15}
$$
x_{j_{\ell+1}}\l\xi,\partial_x\r\frac{\partial^{k+\ell}\psi}{\partial x^{i_1}\dots\partial x^{i_k}\partial \xi^{j_1}\dots\partial \xi^{j_\ell}}
=x_{j_{\ell+1}}\frac{\partial^{k+\ell}(\l\xi,\partial_x\r\psi)}{\partial x^{i_1}\dots\partial x^{i_k}\partial \xi^{j_1}\dots\partial \xi^{j_\ell}}
+\chi_{i_1\dots i_k j_1\dots j_{\ell+1}},
$$
where
$$
\chi_{i_1\dots i_k j_1\dots j_{\ell+1}}=-\ell\,x_{j_{\ell+1}}\sigma(j_1\dots j_\ell)\,
\frac{\partial^{k+\ell}\psi}{\partial x^{i_1}\dots\partial x^{i_k}\partial x^{j_1}
\partial \xi^{j_2}\dots\partial \xi^{j_\ell}}.
$$
By the induction hypothesis,
\begin{equation}
\chi_{i_1\dots i_k j_1\dots j_{\ell+1}}|_{T{\mathbb S}^{n-1}}\in {\mathcal S}(T{\mathbb S}^{n-1}).
                                         \label{Eq2.17}
\end{equation}
Formula \eqref{Eq2.16} becomes now
$$
\begin{aligned}
&\frac{\partial^{k+\ell+1}\psi}{\partial x^{i_1}\dots\partial x^{i_k}\partial \xi^{j_1}\dots\partial \xi^{j_{\ell+1}}}
=\tilde\Xi_{j_{\ell+1}}\frac{\partial^{k+\ell}\psi}{\partial x^{i_1}\dots\partial x^{i_k}\partial \xi^{j_1}\dots\partial \xi^{j_\ell}}\\
&+x_{j_{\ell+1}}\frac{\partial^{k+\ell}(\l\xi,\partial_x\r\psi)}{\partial x^{i_1}\dots\partial x^{i_k}\partial \xi^{j_1}\dots\partial \xi^{j_\ell}}
+\chi_{i_1\dots i_k j_1\dots j_{\ell+1}}
+(\lambda-\ell)\xi_{j_{\ell+1}}\frac{\partial^{k+\ell}\psi}{\partial x^{i_1}\dots\partial x^{i_k}\partial \xi^{j_1}\dots\partial \xi^{j_\ell}}.
\end{aligned}
$$
Taking the restriction to $T{\mathbb S}^{n-1}$, we have
\begin{equation}
\begin{aligned}
\left.\frac{\partial^{k+\ell+1}\psi}{\partial x^{i_1}\dots\partial x^{i_k}\partial \xi^{j_1}\dots\partial \xi^{j_{\ell+1}}}\right|_{T{\mathbb S}^{n-1}}
&=\left.\Big(\tilde\Xi_{j_{\ell+1}}\frac{\partial^{k+\ell}\psi}{\partial x^{i_1}\dots\partial x^{i_k}\partial \xi^{j_1}\dots\partial \xi^{j_\ell}}
\Big)\right|_{T{\mathbb S}^{n-1}}\\
&+x_{j_{\ell+1}}
\left.\frac{\partial^{k+\ell}(\l\xi,\partial_x\r\psi)}{\partial x^{i_1}\dots\partial x^{i_k}\partial \xi^{j_1}\dots\partial \xi^{j_\ell}}
\right|_{T{\mathbb S}^{n-1}}\\
&+\chi_{i_1\dots i_k j_1\dots j_{\ell+1}}|_{T{\mathbb S}^{n-1}}\\
&+\left.(\lambda-\ell)\xi_{j_{\ell+1}}\frac{\partial^{k+\ell}\psi}{\partial x^{i_1}\dots\partial x^{i_k}\partial \xi^{j_1}\dots\partial \xi^{j_\ell}}
\right|_{T{\mathbb S}^{n-1}}.
\end{aligned}
                                         \label{Eq2.18}
\end{equation}

By the induction hypothesis,
$$
\left.\frac{\partial^{k+\ell}\psi}{\partial x^{i_1}\dots\partial x^{i_k}\partial \xi^{j_1}\dots\partial \xi^{j_\ell}}
\right|_{T{\mathbb S}^{n-1}}\in {\mathcal S}(T{\mathbb S}^{n-1}).
$$
Therefore
$$
\left.\Big(\tilde\Xi_{j_{\ell+1}}\frac{\partial^{k+\ell}\psi}{\partial x^{i_1}\dots\partial x^{i_k}\partial \xi^{j_1}\dots\partial \xi^{j_\ell}}
\Big)\right|_{T{\mathbb S}^{n-1}}=
\Xi_{j_{\ell+1}}\Bigg(\left.\frac{\partial^{k+\ell}\psi}{\partial x^{i_1}\dots\partial x^{i_k}\partial \xi^{j_1}\dots\partial \xi^{j_\ell}}
\right|_{T{\mathbb S}^{n-1}}\Bigg).
$$
Formula \eqref{Eq2.18} becomes
$$
\begin{aligned}
\left.\frac{\partial^{k+\ell+1}\psi}{\partial x^{i_1}\dots\partial x^{i_k}\partial \xi^{j_1}\dots\partial \xi^{j_{\ell+1}}}\right|_{T{\mathbb S}^{n-1}}
&=\Xi_{j_{\ell+1}}\Bigg(\left.\frac{\partial^{k+\ell}\psi}{\partial x^{i_1}\dots\partial x^{i_k}\partial \xi^{j_1}\dots\partial \xi^{j_\ell}}
\right|_{T{\mathbb S}^{n-1}}\Bigg)\\
&+x_{j_{\ell+1}}
\left.\frac{\partial^{k+\ell}(\l\xi,\partial_x\r\psi)}{\partial x^{i_1}\dots\partial x^{i_k}\partial \xi^{j_1}\dots\partial \xi^{j_\ell}}
\right|_{T{\mathbb S}^{n-1}}\\
&+\chi_{i_1\dots i_k j_1\dots j_{\ell+1}}|_{T{\mathbb S}^{n-1}}\\
&+\left.(\lambda-\ell)\xi_{j_{\ell+1}}\frac{\partial^{k+\ell}\psi}{\partial x^{i_1}\dots\partial x^{i_k}\partial \xi^{j_1}\dots\partial \xi^{j_\ell}}
\right|_{T{\mathbb S}^{n-1}}.
\end{aligned}
$$
The second and third terms on the right hand side belong to ${\mathcal S}(T{\mathbb S}^{n-1})$ by \eqref{Eq2.11} and \eqref{Eq2.17} respectively.
The last term on the right-hand side belongs to ${\mathcal S}(T{\mathbb S}^{n-1})$ by the induction hypothesis. Finally, the first term on the right-hand side belongs to ${\mathcal S}(T{\mathbb S}^{n-1})$ since $\Xi_{j_{\ell+1}}$ is an intrinsic operator on $T{\mathbb S}^{n-1}$. Thus
$$
\left.\frac{\partial^{k+\ell+1}\psi}{\partial x^{i_1}\dots\partial x^{i_k}\partial \xi^{j_1}\dots\partial \xi^{j_{\ell+1}}}\right|_{T{\mathbb S}^{n-1}}
\in{\mathcal S}(T{\mathbb S}^{n-1}).
$$
This finishes the induction step.
\end{proof}

\subsection{Necessity}

\begin{proof}[Proof of the necessity part in Theorem \ref{Th1.3}]
Given a tensor field $f\in{\mathcal S}({\R}^n;S^m{\R}^n)$, let
$
\varphi^k=I^k\!f\in{\mathcal S}(T{\mathbb S}^{n-1})\ (k=0,1,\dots,m).
$
The definition \eqref{Eq1.10a} implies
$$
\varphi^k(x,-\xi)=(-1)^m\int\limits_{-\infty}^\infty t^k\,\l f(x-t\xi),\xi^m\r\,dt
$$
Changing the integration variable as $t:=-t$, we obtain
$$
\varphi^k(x,-\xi)=(-1)^{m-k}\int\limits_{-\infty}^\infty t^k\,\l f(x+t\xi),\xi^m\r\,dt=(-1)^{m-k}\varphi^k(x,\xi).
$$
This proves property (1) of Theorem \ref{Th1.3}.

Let the function $\psi^m\in C^\infty\big({\R}^n\times({\R}^n\setminus\{0\})\big)$ be defined by \eqref{Eq1.14} with $k=m$. This means
\begin{equation}
\psi^m=J^m\!f.
\label{Eq2.19}
\end{equation}

Observe that $J^m=J^{m,m}$ as is seen from \eqref{Eq2.6} and \eqref{Eq1.11}.
By Lemma \ref{L2.4},
$$
J_{i_1j_1}\dots J_{i_{m+1}j_{m+1}}(J^m\!f)=0\quad\mbox{for all}\quad 1\le i_1,j_1,\dots i_{m+1},j_{m+1}\le n.
$$
Together with \eqref{Eq2.19}, this gives \eqref{Eq1.15}.
\end{proof}

\subsection{Main Lemma}
The following statement is the most essential part of the proof of Theorem \ref{Th1.3}. 

\begin{lemma} \label{S2.7}
Given an integer $m\ge0$, let a function $\psi^m\in C^\infty\big({\R}^n\times({\R}^n\setminus\{0\})\big)$ satisfy \eqref{Eq1.15} and
\begin{equation}
\psi^m(x,t\xi)=\frac{1}{|t|}\,\psi^m(x,\xi)\quad\mbox{for}\quad 0\neq t\in\R.
                                         \label{Eq2.28}
\end{equation}
Define the functions $\psi_{i_1\dots i_m}\in C^\infty\big({\R}^n\times({\R}^n\setminus\{0\})\big)$ for all indices $1\le i_1,\dots,i_m\le n$ by
\begin{equation}
\psi_{i_1\dots i_m}=\frac{(-1)^m}{m!}\sigma(i_1\dots i_m)\Big(\sum\limits_{k=0}^m\frac{1}{(m-k)!}\,\frac{\partial^m}
{\partial x^{i_1}\dots\partial x^{i_k}\partial\xi^{i_{k+1}}\dots\partial\xi^{i_m}}\l\xi,\partial_x\r^{m-k}\Big)\psi^m.
                                         \label{Eq2.22}
\end{equation}
Then
\begin{equation}
J_{ij}\psi_{i_1\dots i_m}=0\quad\mbox{for all}\quad 1\le i,j\le n.
                                         \label{Eq2.29}
\end{equation}
\end{lemma}

\bpr
The John operator $J_{ij}$ commutes with partial derivatives (as any differential operator with constant coefficients). Therefore \eqref{Eq2.29} is equivalent to the equation
$$
\sigma(i_1\dots i_m)\Big(\sum\limits_{k=0}^m\frac{1}{(m-k)!}\,\frac{\partial^m}
{\partial x^{i_1}\dots\partial x^{i_k}\partial\xi^{i_{k+1}}\dots\partial\xi^{i_m}}J_{ij}\l\xi,\partial_x\r^{m-k}\Big)\psi^m=0.
$$
Moreover, the operators $\l\xi,\partial_x\r$ and $J_{ij}$ commute as one can easily check. Therefore our equation takes the form
\begin{equation}
\sigma(i_1\dots i_m)\Big(\sum\limits_{k=0}^m\frac{1}{(m-k)!}\,\frac{\partial^m}
{\partial x^{i_1}\dots\partial x^{i_k}\partial\xi^{i_{k+1}}\dots\partial\xi^{i_m}}\l\xi,\partial_x\r^{m-k}\Big)J_{ij}\psi^m=0.
                                         \label{Eq2.30}
\end{equation}
Thus we have to prove that \eqref{Eq2.28} and \eqref{Eq1.15} imply \eqref{Eq2.30}. This is trivially true for $m=0$.
We proceed by induction in $m$.
Assume Lemma \ref{S2.7} to be valid for some $m\ge0$. Now, let a function $\psi^{m+1}\in C^\infty\big({\R}^n\times({\R}^n\setminus\{0\})\big)$ satisfy
\begin{equation}
\psi^{m+1}(x,t\xi)=\frac{1}{|t|}\psi^{m+1}(x,\xi)\quad(0\neq t\in{\R})
                                         \label{Eq2.31}
\end{equation}
and
\begin{equation}
J_{ij}J_{i_1j_1}\dots J_{i_{m+1}j_{m+1}}\psi^{m+1}=0\quad\mbox{for all}\quad 1\le i,j,i_1,j_1,\dots,i_{m+1},j_{m+1}\le n.
                                         \label{Eq2.32}
\end{equation}
We have to prove that $\psi^{m+1}$ satisfies equation \eqref{Eq2.30} with $m$ increased by 1, i.e.,
\begin{equation}
\sigma(i_1\dots i_{m+1})\Big(\sum\limits_{k=0}^m\frac{1}{(m-k+1)!}\,\frac{\partial^{m+1}}
{\partial x^{i_1}\dots\partial x^{i_k}\partial\xi^{i_{k+1}}\dots\partial\xi^{i_{m+1}}}\l\xi,\partial_x\r^{m-k+1}\Big)J_{ij}\psi^{m+1}=0
                                         \label{Eq2.33}
\end{equation}
for all $1\le i,j,i_1,\dots i_{m+1}\le n$.

Let us temporary fix values of indices $(i_1,j_1,\dots,i_{m+1},j_{m+1})$ and set $ \wt{J}= 	J_{i_1j_1}\dots J_{i_{m+1} j_{m+1}}$. Equation \eqref{Eq2.32} is now written as $J_{ij}\widetilde{J}\psi^{m+1}=0$ or
$$
\frac{\PD^2 (\widetilde J\psi^{m+1}) }{\PD x^i\PD \xi^j}- \frac{\PD^2(\widetilde J\psi^{m+1})}{\PD x^j\PD \xi^i}=0.
$$
Multiplying this equation by $ \xi^j $ and taking sum over $j$, we have
$$
\xi^j \frac{\PD^2 (\widetilde{J}\psi^{m+1}) }{\PD x^i \PD \xi^j}-  \xi^j \frac{\PD^2(\widetilde{J}\psi^{m+1})}{\PD x^j\PD \xi^i}=0.
$$
This can be written in the form
\Beq
\frac{\PD}{\PD x^i} \l \xi,\PD_{\xi}\r \wt{J}\psi^{m+1} - \frac{\PD}{\PD\xi^i} \l \xi,\PD_{x}\r \wt{J}\psi^{m+1}+ \frac{\PD}{\PD x^i} \wt{J}\psi^{m+1}=0.
                                            \label{Eq2.34}
\Eeq
Note that $ \wt{J}\psi^{m+1}(x,\xi) $ is positively homogeneous of degree $ -(m+2) $ in its second argument. By the Euler equation for homogeneous functions,
$$
\l \xi,\PD_{\xi}\r \wt{J}\psi^{m+1} = -(m+2)  \wt{J}\psi^{m+1}.
$$
Substituting this expression into \eqref{Eq2.34}, we get
\[
(m+1)\frac{\PD}{\PD x^i}\wt{J} \psi^{m+1}+\frac{\PD}{\PD \xi^i}\l \xi,\PD_{x}\r \wt{J}\psi^{m+1}=0.
\]
Since $ \l \xi,\PD_{x}\r $ commutes with $ \wt{J} $, the last equation can be written as
$$
	 \wt{J} \left(\frac{1}{m+1}\,\frac{\PD}{\PD\xi^i} \l \xi,\PD_{x}\r + \frac{\PD}{\PD x^i} \right) \psi^{m+1}=0.
$$
Substituting the value $ \wt{J}= 	J_{i_1j_1}\dots J_{i_{m+1} j_{m+1}}$, we write this in the final form
\begin{equation}
J_{i_1j_1}\dots J_{i_{m+1} j_{m+1}}\left(\frac{1}{m+1}\,\frac{\PD}{\PD\xi^i} \l \xi,\PD_{x}\r + \frac{\PD}{\PD x^i} \right) \psi^{m+1}=0.
                                         \label{Eq2.35}
\end{equation}
From now on, the indices $(i_1,j_1,\dots,i_{m+1},j_{m+1})$ take again arbitrary values. Thus \eqref{Eq2.35} holds for all
$1\le i,i_1,j_1,\dots,i_{m+1},j_{m+1}\le n$.

For every index $i_{m+1}$ satisfying $1\le i_{m+1}\le n$, we define the function $\psi^m_{i_{m+1}}$ by
\begin{equation}
\psi^m_{i_{m+1}}(x,\xi)=\left(\frac{1}{m+1}\,\frac{\PD}{\PD\xi^{i_{m+1}}} \l \xi,\PD_{x}\r + \frac{\PD}{\PD x^{i_{m+1}}} \right) \psi^{m+1}(x,\xi).
                                        \label{Eq2.36}
\end{equation}
As easily follows from \eqref{Eq2.31} and \eqref{Eq2.36}, this function satisfies
\begin{equation}
\psi^m_{i_{m+1}}(x,t\xi)=\frac{1}{|t|}\psi^m_{i_{m+1}}(x,\xi)\quad(0\neq t\in{\R}).
                                         \label{Eq2.37}
\end{equation}
Equations  \eqref{Eq2.35} and \eqref{Eq2.37} mean that, for every ${i_{m+1}}$, the function $\psi^m_{i_{m+1}}$ satisfies hypotheses of Lemma \ref{S2.7}. By the induction hypothesis \eqref{Eq2.30}, for all $1\le i,j,{i_{m+1}},i_1,\dots, i_m\le n$,
\begin{align}
\sigma(i_1\dots i_m)\Big(\sum\limits_{k=0}^m\frac{1}{(m-k)!}\,\frac{\partial^m}
	{\partial x^{i_1}\dots\partial x^{i_k}\partial\xi^{i_{k+1}}\dots\partial\xi^{i_m}}\l\xi,\partial_x\r^{m-k}\Big)J_{ij}\psi^m_{i_{m+1}}=0.
                       \label{Eq2.38}
\end{align}

Let us denote
\Beq
\chi^{m}_{i_1\dots i_{m+1}}=\sigma(i_1\dots i_m)\Big(\sum\limits_{k=0}^m\frac{1}{(m-k)!}\,\frac{\partial^m}
{\partial x^{i_1}\dots\partial x^{i_k}\partial\xi^{i_{k+1}}\dots\partial\xi^{i_m}}\l\xi,\partial_x\r^{m-k}\Big)\psi^m_{i_{m+1}}.
                                       \label{Eq2.39}
\Eeq
Equation \eqref{Eq2.38} is written in terms of functions $\chi^{m}_{i_1\dots i_{m+1}}$ as follows:
\begin{align}
J_{ij}\chi^{m}_{i_1\dots i_{m+1}}=0.
                       \label{Eq2.40}
\end{align}
We will show the following equality, which would complete the induction step and hence the proof of Lemma \ref{S2.7}:
\Beq
\chi^{m}_{i_1\!\dots i_{m+1}}=\sigma(i_1\!\dots i_{m+1})\Big( \sum\limits_{k=0}^{m+1}\frac{1}{(m\!-\!k\!+\!1)!}\,\frac{\partial^{m+1}}
{\partial x^{i_1}\!\dots\partial x^{i_k}\partial\xi^{i_{k\!+\!1}}\!\dots\partial\xi^{i_m}\PD\xi^{i_{m\!+\!1}}}
\l\xi,\partial_x\r^{m\!-\!k\!+\!1}\Big)\psi^{m\!+\!1}.
                                       \label{Eq2.41}
\Eeq
Indeed, substitution of \eqref{Eq2.41} into \eqref{Eq2.40} gives \eqref{Eq2.33}.
By the way, formula \eqref{Eq2.41} shows that $\chi^{m}_{i_1\dots i_{m+1}}$ is symmetric in the indices $(i_1,\dots, i_{m+1})$. The latter fact is not obvious from definition \eqref{Eq2.39}.

It remains to prove \eqref{Eq2.41}. To this end we substitute expression \eqref{Eq2.36} for $\psi^m_{i_{m+1}}$ into \eqref{Eq2.39}
\Beq
\begin{aligned}
\chi^{m}_{i_1\!\dots i_{m+1}}=\sigma(i_1\dots i_m)&\Bigg(\sum\limits_{k=0}^m\frac{1}{(m-k)!}\,\frac{\partial^m}
{\partial x^{i_1}\dots\partial x^{i_k}\partial\xi^{i_{k+1}}\dots\partial\xi^{i_m}}\l\xi,\partial_x\r^{m-k}\Bigg)\times \\
&\times\left( \frac{1}{m+1}\frac{\PD}{\PD\xi^{i_{m+1}}} \l \xi,\PD_{x}\r + \frac{\PD}{\PD x^{i_{m+1}}} \right)\psi^{m+1}.
\end{aligned}
                                       \label{Eq2.42}
\Eeq

The commutator formula
	\[
	\l \xi, \PD_x\r^{m-k} \frac{\PD}{\PD \xi^{i_{m+1}}}=\frac{\PD}{\PD \xi^{i_{m+1}}}\l \xi, \PD_x\r^{m-k} - (m-k) \l \xi, \PD_x\r^{m-k-1}\frac{\PD}{\PD x^{i_{m+1}}}
	\]
is a partial case of Lemma \ref{L2.2}. With the help of the latter formula, \eqref{Eq2.42} becomes
	\begin{align*}
\chi^{m}_{i_1\!\dots i_{m+1}}= \sigma(i_1\dots i_m)\Bigg[&\sum\limits_{k=0}^m\frac{1}{(m-k)!}\,\frac{\partial^{m}}
	{\partial x^{i_1}\dots\partial x^{i_k}\partial\xi^{i_{k+1}}\dots\partial\xi^{i_m}}  \Bigg(\frac{1}{m+1}\frac{\PD}{\PD\xi^{i_{m+1}}} \l\xi,\partial_x\r^{m-k+1} \\
	&+ \l \xi,\PD_{x}\r^{m-k}\lb 1- \frac{m-k}{m+1}\rb \frac{\PD}{\PD x^{i_{m+1}}}\Bigg) \Bigg]\psi^{m+1}.
\end{align*}
This is easily transformed to the form
\begin{align*}
\chi^{m}_{i_1\!\dots i_{m+1}}=	\frac{\sigma(i_1\dots i_{m})}{m+1}&\Bigg[ \sum\limits_{k=0}^m\frac{m-k+1}{(m-k+1)!}\,\frac{\partial^{m+1}}
	{\partial x^{i_1}\dots\partial x^{i_k}\partial\xi^{i_{k+1}}\dots\PD\xi^{i_{m+1}}}\l\xi,\partial_x\r^{m-k+1}\\&+ \sum\limits_{k=0}^m\frac{k+1}{(m-k)!}\,\frac{\partial^{m+1}}
	{\partial x^{i_1}\dots\partial x^{i_k}\PD x^{i_{m+1}}\partial\xi^{i_{k+1}}\dots\partial\xi^{i_m}}\l\xi,\partial_x\r^{m-k} \Bigg]\psi^{m+1}.
\end{align*}
Replacing the summation variable $ k $ with $ k-1 $ in the second sum, we get
$$
\begin{aligned}
	\chi^{m}_{i_1\dots i_{m+1}}=&\frac{\sigma(i_1\dots i_{m})}{m+1}\Bigg[ \sum\limits_{k=0}^m\frac{m-k+1}{(m-k+1)!}\,\frac{\partial^{m+1}}
	{\partial x^{i_1}\dots\partial x^{i_k}\partial\xi^{i_{k+1}}\dots\PD\xi^{i_{m+1}}}\l\xi,\partial_x\r^{m-k+1}\\
	 &+ \sum\limits_{k=1}^{m+1}\frac{k}{(m-k+1)!}\,\frac{\partial^{m+1}}
	{\partial x^{i_1}\dots\partial x^{i_{k-1}}\PD x^{i_{m+1}}\partial\xi^{i_{k}}\dots\partial\xi^{i_m}}\l\xi,\partial_x\r^{m-k+1} \Bigg]\psi^{m+1}.
\end{aligned}
$$
Both summations can be extended to the limits $0\le k\le m+1$ because of the presence of the factors $m-k+1$ and $k$ in the first and second sum respectively. In this way, our formula becomes
\begin{equation}
\begin{aligned}
	\chi^{m}_{i_1\dots i_{m+1}}=\frac{1}{m+1}\sum\limits_{k=0}^{m+1}\Bigg[& \frac{\sigma(i_1\dots i_{m})}{(m-k+1)!}\left((m-k+1)\,\frac{\partial^{m+1}}
	{\partial x^{i_1}\dots\partial x^{i_k}\partial\xi^{i_{k+1}}\dots\PD\xi^{i_{m+1}}}\right.\\
	 &+\left. k\,\frac{\partial^{m+1}}
	{\partial x^{i_1}\dots\partial x^{i_{k-1}}\PD x^{i_{m+1}}\partial\xi^{i_{k}}\dots\partial\xi^{i_m}}\right) \Bigg]\l\xi,\partial_x\r^{m-k+1}\psi^{m+1}.
\end{aligned}
                                            \label{Eq2.43}
\end{equation}

Let us recall an easy statement of tensor algebra. If a tensor $(y_{i_1\dots i_{m+1}})$ is symmetric in first $k$ indices and in last $m-k+1$ indices, then
$$
\sigma(i_1\dots i_{m+1})y_{i_1\dots i_{m+1}}=\frac{1}{m+1}\,\sigma(i_1\dots i_{m})[(m-k+1)y_{i_1\dots i_{m+1}}+ky_{i_{m+1}i_1\dots i_{m}}].
$$
The proof is given in \cite[Lemma 2.4.1]{Sh} although the reader can easily prove this by themselves. We apply the statement to the operator-valued tensor
$$
y_{i_1\dots i_{m+1}}=\frac{\partial^{m+1}}{\partial x^{i_1}\dots\partial x^{i_k}\partial\xi^{i_{k+1}}\dots\PD\xi^{i_{m+1}}}
$$
to obtain
$$
\begin{aligned}
\sigma(i_1\dots i_{m})\Big((m\!-\!k\!+\!1)\,&\frac{\partial^{m+1}}{\partial x^{i_1}\!\dots\partial x^{i_k}\partial\xi^{i_{k+1}}\!\dots\PD\xi^{i_{m+1}}}
+k\,\frac{\partial^{m+1}}{\partial x^{i_1}\!\dots\partial x^{i_{k-1}}\PD x^{i_{m+1}}\partial\xi^{i_{k}}\!\dots\partial\xi^{i_m}}\Big)=\\
&=(m+1)\sigma(i_1\dots i_{m+1})\frac{\partial^{m+1}}{\partial x^{i_1}\!\dots\partial x^{i_k}\partial\xi^{i_{k+1}}\!\dots\PD\xi^{i_{m+1}}}.
\end{aligned}
$$
Substituting this value into \eqref{Eq2.43}, we arrive at \eqref{Eq2.41}.
\end{proof}

\subsection{Sufficiency}
We start the proof of the sufficiency part in Theorem \ref{Th1.3}. The proof consists of several
steps that are called \emph{statements}. Hypotheses of statements coincide with that of Theorem \ref{Th1.3}.

Given functions
$
\varphi^k\in{\mathcal S}(T{\mathbb S}^{n-1})\quad(k=0,\dots,m)
$
satisfying \eqref{Eq1.14a},
let the function $\psi^m\in C^\infty\big({\R}^n\times({\R}^n\setminus\{0\})\big)$ be defined by \eqref{Eq1.14} with $k=m$. Assume the function to satisfy
\eqref{Eq1.15}. We also define the functions
$\psi^k\in C^\infty\big({\R}^n\times({\R}^n\setminus\{0\})\big)\ (k=0,\dots,m-1)$ by \eqref{Eq1.14}.
Thus, \eqref{Eq1.14} is valid for all $k=0,1,\dots,m$. It implies
\Beq           \label{Eq2.21}
\psi^{k}|_{T\Sb^{n-1}}=\vp^{k}\quad \mbox{for}\quad 0\leq k\leq m.
\Eeq
In Section 1, properties \eqref{Eq1.12}--\eqref{Eq1.14} were easily derived from \eqref{Eq1.11}. Now we have no tensor field $f$ (we are, in fact, proving the existence of such a tensor field) and we cannot use Formula \eqref{Eq1.11}. Nevertheless, the functions $\psi^k$ possess the same properties as the following statement shoes.

\begin{statement} \label{S2.6}
For every $k=0,\dots,m$,
\begin{equation}
\psi^k(x,t\xi)=\frac{t^{m-k}}{|t|}\,\psi^k(x,\xi)\quad\mbox{for}\quad 0\neq t\in\R
                                         \label{Eq2.23}
\end{equation}
and
\begin{equation}
\psi^k(x+t\xi,\xi)=\sum\limits_{\ell=0}^k{k\choose \ell}(-t)^{k-\ell}\,\psi^\ell(x,\xi)\quad\mbox{for}\quad  t\in\R.
                                         \label{Eq2.24}
\end{equation}
\end{statement}

\begin{proof}
For $t\ne0$ by \eqref{Eq1.14},
\begin{equation}
\psi^k(x,t\xi)=|t|^{m-2k-1}|\xi|^{m-2k-1}\!\sum\limits_{\ell=0}^k(-1)^{k-\ell}{k\choose \ell}|t|^\ell|\xi|^\ell t^{k-\ell}\l\xi,x\r^{k-\ell}\,\varphi^\ell
\Big(x-\frac{\l\xi,x\r}{|\xi|^2}\xi,\mbox{sgn}(t)\,\frac{\xi}{|\xi|}\Big).
                                         \label{Eq2.25}
\end{equation}
In the case of $t>0$, this reads
$$
\psi^k(x,t\xi)=t^{m-k-1}|\xi|^{m-2k-1}\sum\limits_{\ell=0}^k(-1)^{k-\ell}{k\choose \ell}|\xi|^\ell \l\xi,x\r^{k-\ell}\,\varphi^\ell
\Big(x-\frac{\l\xi,x\r}{|\xi|^2}\xi,\frac{\xi}{|\xi|}\Big).
$$
Together with \eqref{Eq1.14}, this gives
$$
\psi^k(x,t\xi)=t^{m-k-1}\psi^k(x,\xi).
$$
This coincides with \eqref{Eq2.23} for $t>0$.

In the case of $t<0$, \eqref{Eq2.25} reads
$$
\psi^k(x,t\xi)=(-t)^{m-k-1}|\xi|^{m-2k-1}\sum\limits_{\ell=0}^k{k\choose \ell}|\xi|^\ell \l\xi,x\r^{k-\ell}\,\varphi^\ell
\Big(x-\frac{\l\xi,x\r}{|\xi|^2}\xi,-\frac{\xi}{|\xi|}\Big).
$$
On using \eqref{Eq1.14a}, this takes the form
$$
\psi^k(x,t\xi)=-t^{m-k-1}|\xi|^{m-2k-1}\sum\limits_{\ell=0}^k(-1)^{k-\ell}{k\choose \ell}|\xi|^\ell \l\xi,x\r^{k-\ell}\,\varphi^\ell
\Big(x-\frac{\l\xi,x\r}{|\xi|^2}\xi,\frac{\xi}{|\xi|}\Big).
$$
Together with \eqref{Eq1.14}, this gives
$$
\psi^k(x,t\xi)=-t^{m-k-1}\psi^k(x,\xi).
$$
This coincides with \eqref{Eq2.23} for $t<0$.

By \eqref{Eq1.14},
$$
\psi^k(x+t\xi,\xi)=|\xi|^{m-2k-1}\sum\limits_{\ell=0}^k(-1)^{k-\ell}{k\choose \ell}|\xi|^\ell (\l\xi,x\r+t|\xi|^2)^{k-\ell}\,\varphi^\ell
\Big(x-\frac{\l\xi,x\r}{|\xi|^2}\xi,\frac{\xi}{|\xi|}\Big).
$$
Expanding $(\l\xi,x\r+t|\xi|^2)^{k-\ell}$ in powers of $t$, we write this in the form
$$
\psi^k(x+t\xi,\xi)=|\xi|^{m-2k-1}\sum\limits_{\ell=0}^k(-1)^{k-\ell}{k\choose \ell}\sum\limits_{s=0}^{k-\ell}{{k-\ell}\choose s} |\xi|^{\ell+2s}\l\xi,x\r^{k-\ell-s} \,t^s\,\varphi^\ell\Big(x-\frac{\l\xi,x\r}{|\xi|^2}\xi,\frac{\xi}{|\xi|}\Big).
$$
After changing the order of summations, this becomes
\begin{equation}
\psi^k(x+t\xi,\xi)=\sum\limits_{s=0}^k t^s\sum\limits_{\ell=0}^{k-s}(-1)^{k-\ell}{k\choose \ell}{{k\!-\!\ell}\choose s}
|\xi|^{m-2k+\ell+2s+1}\l\xi,x\r^{k-\ell-s} \,\varphi^\ell\Big(x-\frac{\l\xi,x\r}{|\xi|^2}\xi,\frac{\xi}{|\xi|}\Big).
                                         \label{Eq2.26}
\end{equation}

Let us transform the right hand side of \eqref{Eq2.26}. Obviously,
$$
\sum\limits_{\ell=0}^k{k\choose \ell}(-t)^{k-\ell}\,\psi^\ell(x,\xi)=
\sum\limits_{s=0}^{{k}}(-1)^s{k\choose s}t^s\,\psi^{k-s}(x,\xi)
$$
(we just changed the summation variable as $\ell=k-s$). Substituting value \eqref{Eq1.14} for $\psi^{k-s}(x,\xi)$ into the right-hand side of the last formula, we obtain
$$
\begin{aligned}
&\sum\limits_{\ell=0}^k{k\choose \ell}(-t)^{k-\ell}\,\psi^\ell(x,\xi)=\\
&=\sum\limits_{s=0}^k(-1)^s{k\choose s}t^s|\xi|^{m-2(k-s)-1}\sum\limits_{\ell=0}^{k-s}(-1)^{k-s-\ell}{{k-s}\choose \ell}|\xi|^\ell\,
\l\xi,x\r^{k-s-\ell}\varphi^\ell\Big(x-\frac{\l\xi,x\r}{|\xi|^2}\xi,\frac{\xi}{|\xi|}\Big).
\end{aligned}
$$
This can be written as
\begin{equation}
\begin{aligned}
&\sum\limits_{\ell=0}^k{k\choose \ell}(-t)^{k-\ell}\,\psi^\ell(x,\xi)=\\
&=\sum\limits_{s=0}^k t^s\sum\limits_{\ell=0}^{k-s}(-1)^{k\!-\!\ell}{k\choose s}{{k-s}\choose \ell}|\xi|^{m-2k+2s+\ell-1}\,
\l\xi,x\r^{k-s-\ell}\varphi^\ell\Big(x-\frac{\l\xi,x\r}{|\xi|^2}\xi,\frac{\xi}{|\xi|}\Big).
\end{aligned}
                                         \label{Eq2.27}
\end{equation}
In virtue of the obvious equality
$$
{k\choose \ell}{{k-\ell}\choose s}={k\choose s}{{k-s}\choose \ell},
$$
right hand sides of formulas \eqref{Eq2.26} and \eqref{Eq2.27} coincide. Equating left hand sides of these formulas, we obtain \eqref{Eq2.24}.
\end{proof}

Equations \eqref{Eq1.15} and \eqref{Eq2.23} mean that the function $\psi^m$ satisfies hypotheses of Lemma \ref{S2.7}. Applying the lemma, we can state that the functions $\psi_{i_1\dots i_m}\in C^\infty\big({\R}^n\times({\R}^n\setminus\{0\})\big)$, defined by \eqref{Eq2.22} for all indices $1\le i_1,\dots,i_m\le n$,  satisfy \eqref{Eq2.29}. To study properties of functions $\psi_{i_{1}\dots i_{m}}$, it is convenient to look at an alternate formulation of the formula \eqref{Eq2.22} which we now show.

\begin{statement}\label{S2.8}
For all indices $(i_1,\dots,i_m)$,
\begin{equation}
\psi_{i_1\dots i_m}=\frac{1}{m!}\sigma(i_1\dots i_m)\sum\limits_{k=0}^m(-1)^k{m\choose k}\frac{\partial^m \psi^k}
		{\partial x^{i_1}\dots\partial x^{i_k}\partial\xi^{i_{k+1}}\dots\partial\xi^{i_m}}.
                                         \label{Eq2.44}
\end{equation}
\end{statement}

Compare \eqref{Eq2.44} with the inversion formula \cite[Formula (3.1)]{KMSS}.

\begin{proof}
We first show that for every $k=0,1,\dots,m$ and for every integer $\ell\ge0$,
\begin{equation}
\l\xi,\partial_x\r^\ell\psi^k=\left\{
\begin{aligned}
&(-1)^\ell{k\choose \ell}\ell!\,\psi^{k-\ell}\quad\mbox{if}\quad\ell\le k,\\
&0\quad\mbox{if}\quad\ell>k.
\end{aligned}\right.
                                         \label{Eq2.45}
\end{equation}
To see \eqref{Eq2.45}, we first change the summation variable in equation \eqref{Eq2.24} as $\ell=k-p$. We get,
$$
\psi^k(x+t\xi,\xi)=\sum\limits_{p=0}^k(-1)^p{k\choose p}t^p\,\psi^{k-p}(x,\xi).
$$
Apply the operator $\l\xi,\partial_x\r^\ell$ to this equation
$$
(\l\xi,\partial_x\r^\ell\psi^k)(x+t\xi,\xi)=\frac{\partial^\ell\big(\psi^k(x+t\xi,\xi)\big)}{\partial t^\ell}
=\left\{
\begin{aligned}
&\sum\limits_{p=\ell}^k(-1)^\ell{k\choose p}\frac{p!}{(p\!-\!\ell)!}t^{p-\ell}\,\psi^{k-p}(x,\xi)\ \mbox{if}\ \ell\le k,\\
&0\quad\mbox{if}\quad\ell>k.
\end{aligned}\right.
$$
Setting $t=0$, we obtain \eqref{Eq2.45}.

Now from \eqref{Eq2.45}, we have
\[
\langle \xi,\PD_x\rangle^{m-k} \psi^{m}=(-1)^{m-k}{m\choose k}(m-k)!\psi^{k}.
\]
Substituting this into \eqref{Eq2.22}, we obtain \eqref{Eq2.44}. Also from \eqref{Eq2.45},
$$
\psi^k=(-1)^{m-k}{m\choose k}^{-1}\frac{1}{(m-k)!}\l\xi,\partial_x\r^{m-k}\psi^m\quad(0\le k\le m).
$$
Substituting this value into \eqref{Eq2.44}, we obtain \eqref{Eq2.22}. Thus formulas \eqref{Eq2.22} and \eqref{Eq2.44} are equivalent.
\end{proof}

\begin{statement} \label{S2.10}
	For every $k=0,\dots,m$,
	\begin{equation}
		\left.\frac{\partial^{r+s}\psi^k}{\partial x^{i_1}\dots\partial x^{i_r}\partial \xi^{j_i}\dots\partial \xi^{j_s}}\right|_{T{\mathbb S}^{n-1}}
		\in{\mathcal S}(T{\mathbb S}^{n-1})\quad\mbox{for all}\quad 1\le i_1,\dots,i_r,j_1,\dots,j_s\le n.
		                            \label{Eq2.46}
	\end{equation}
\end{statement}

\begin{proof}
	By Statement \ref{S2.6}, every function $\psi^k(x,\xi)$ is positively homogeneous in $\xi$. By \eqref{Eq2.21},
	$\psi^k|_{T{\mathbb S}^{n-1}}=\varphi^k\in{\mathcal S}(T{\mathbb S}^{n-1})$.
	
	By \eqref{Eq2.45}, $\l\xi,\partial_x\r\psi^0=0$. Thus the function $\psi^0$ satisfies hypotheses of Lemma \ref{L2.5}. Applying this Lemma, we obtain \eqref{Eq2.46} for $k=0$.
	
	We proceed by induction in $k$. Assume \eqref{Eq2.46} to be valid for some $k$. By \eqref{Eq2.45} again,
	$\l\xi,\partial_x\r\psi^{k+1}=-(k+1)\psi^k$. Together with the induction hypothesis, this implies that $\psi^{k+1}$ satisfies hypotheses of Lemma \ref{L2.5}. Applying this Lemma, we obtain \eqref{Eq2.46} for $k:=k+1$.
\end{proof}

\begin{statement}\label{S2.11}
For every $m$-tuple $(i_1,\dots,i_m)$, the function $\psi_{i_1\dots i_m}$ satisfies
\begin{equation}
\psi_{i_1\dots i_m}(x,t\xi)=\frac{1}{|t|}\psi_{i_1\dots i_m}(x,\xi)\ \mbox{for}\ 0\neq t\in\R
                                         \label{Eq2.47}
\end{equation}
and
\begin{equation}
\psi_{i_1\dots i_m}(x+t\xi,\xi)=\psi_{i_1\dots i_m}(x,\xi)\ \mbox{for}\ t\in\R.
                                         \label{Eq2.47a}
\end{equation}
\end{statement}

\bpr
We say $\chi\in C^\infty\big({\R}^n\times({\R}^n\setminus\{0\})\big)$ is a {\it superhomogeneous} function of degree $\lambda$ if
$\chi(t\xi)=\frac{t^{\lambda+1}}{|t|}\chi(\xi)$ for $0\neq t\in\R$. By \eqref{Eq2.23}, $\psi^k(x,\xi)$ is a superhomogeneous function of degree $m-k-1$ in $\xi$. Every derivative $\frac{\partial}{\partial\xi^i}$ decreases the degree of superhomogeneity by one. Therefore all summands on the right-hand side of \eqref{Eq2.44} are superhomogeneous functions of degree $-1$ in $\xi$. This proves \eqref{Eq2.47}.

The property \eqref{Eq2.47a} is equivalent to $\frac{\partial}{\partial t}\big(\psi_{i_1\dots i_m}(x+t\xi,\xi)\big)=0$. Since
$$
\frac{\partial}{\partial t}\big(\psi_{i_1\dots i_m}(x+t\xi,\xi)\big)=(\l\xi,\partial_x\r\psi_{i_1\dots i_m})(x+t\xi,\xi),
$$
\eqref{Eq2.47a} is equivalent to
\begin{equation}
\l\xi,\partial_x\r\psi_{i_1\dots i_m}=0.
                                         \label{Eq2.48}
\end{equation}

Applying the operator $\l\xi,\partial_x\r$ to the equation \eqref{Eq2.22}, we obtain
$$
\begin{aligned}
\l\xi,\partial_x\r&\psi_{i_1\dots i_m}=\\
&=\frac{(-1)^m}{m!}\Big[\sigma(i_1\dots i_m)\sum\limits_{k=0}^m\frac{1}{(m-k)!}\,\l\xi,\partial_x\r\frac{\partial^m}
{\partial x^{i_1}\dots\partial x^{i_k}\partial\xi^{i_{k+1}}\dots\partial\xi^{i_m}}\l\xi,\partial_x\r^{m-k}\Big]\psi^m.
\end{aligned}
$$
By Corollary \ref{C2.1}, the operator in the brackets coincides with $\frac{1}{m!}\frac{\partial^m}{\partial\xi^{j_1}\dots\partial\xi^{j_m}}\l\xi,\partial_x\r^{m+1}$. Thus,
$$
\l\xi,\partial_x\r\psi_{i_1\dots i_m}=
\frac{(-1)^m}{(m!)^2}\frac{\partial^m}{\partial\xi^{j_1}\dots\partial\xi^{j_m}}\l\xi,\partial_x\r^{m+1}\psi^m.
$$
The right-hand side is equal to zero by \eqref{Eq2.45}. This proves \eqref{Eq2.47a}.
\epr

Together with Statement \ref{S2.10}, \eqref{Eq2.22}, or equivalently, \eqref{Eq2.44} implies that
\begin{equation}
\varphi_{i_1\dots i_m}(x,\xi)=\psi_{i_1\dots i_m}|_{T{\mathbb S}^{n-1}}\in{\mathcal S}(T{\mathbb S}^{n-1}).
                          \label{Eq2.49}
\end{equation}

As easily follows from \eqref{Eq2.47}--\eqref{Eq2.47a},
the functions $\psi_{i_1\dots i_m}$ can be recovered from $\varphi_{i_1\dots i_m}$ by
\begin{equation}
\psi_{i_1\dots i_m}(x,\xi)=\frac{1}{|\xi|}\varphi_{i_1\dots i_m}\Big(x-\frac{\l\xi,x\r}{|\xi|^2}\xi,\frac{\xi}{|\xi|}\Big).
\label{Eq2.50}
\end{equation}
Formulas \eqref{Eq2.29}, \eqref{Eq2.49} and \eqref{Eq2.50} mean that, for every $m$-tuple $(i_1,\dots,i_m)$ of integers satisfying $1\le i_1,\dots,i_m\le n$, the function $\varphi_{i_1\dots i_m}\in{\mathcal S}(T{\mathbb S}^{n-1})$ satisfies the hypotheses of Theorem \ref{Th1.1} with $m=0$. Applying this theorem, we obtain the function $f_{i_1\dots i_m}\in{\mathcal S}({\R}^n)$ such that
\begin{equation}
I^0f_{i_1\dots i_m}=\varphi_{i_1\dots i_m}.
                     \label{Eq2.51}
\end{equation}
We emphasize that the hypothesis $n\ge3$ of Theorem \ref{Th1.3} is essentially used on this step; Theorem \ref{Th1.1} is not true in the case of $n=2$.
Together with \eqref{Eq2.50}, \eqref{Eq2.51} implies
\begin{equation}
J^0f_{i_1\dots i_m}=\psi_{i_1\dots i_m}.
                     \label{Eq2.51a}
\end{equation}
Let $f\in{\mathcal S}({\R}^n;S^m{\R}^n)$ be the symmetric tensor field whose coordinates are $f_{i_1\dots i_m}$.

\begin{statement} \label{S2.12}
	For every $k=0,1,\dots,m$,
	\begin{equation}
	J^k\!f=\psi^k.
	              \label{Eq2.52}
	\end{equation}
\end{statement}

\begin{proof}
	We will prove that
	\begin{equation}
	\frac{\partial^k(J^k\!f-\psi^k)}{\partial x^{j_1}\dots\partial x^{j_k}}=0
	                                  \label{Eq2.53}
	\end{equation}
	for all integers $(j_1,\dots,j_k)$ satisfying $1\le j_1,\dots,j_k\le n$. This will imply \eqref{Eq2.52}. Indeed, for a fixed $0\neq\xi\in{\R}^n$, the restriction of the function $(J^k\!f-\psi^k)(\cdot,\xi)$ to the hyperplane $\xi^\bot=\{x\in{\R}^n\mid\l\xi,x\r=0\}$ belongs to the Schwartz space ${\mathcal S}(\xi^\bot)$. Equation \eqref{Eq2.53} means that all $k$th order derivatives of the latter function are identically equal to zero. This implies that the restriction itself must be identically equal to zero. Since $\xi$ is arbitrary, this proves \eqref{Eq2.52}.
	
	By Lemma \ref{L2.3} with $\ell=k$,
	$$
	\frac{\partial^k(J^k\!f_{i_1\dots i_m})}{\partial x^{j_1}\dots\partial x^{j_k}}=
	\frac{\partial^k(J^0\!f_{i_1\dots i_m})}{\partial \xi^{j_1}\dots\partial \xi^{j_k}}.
	$$
	Together with \eqref{Eq2.51a}, this gives
	\begin{equation}
	\frac{\partial^k(J^k\!f_{i_1\dots i_m})}{\partial x^{j_1}\dots\partial x^{j_k}}=
	\frac{\partial^k\psi_{i_1\dots i_m}}{\partial \xi^{j_1}\dots\partial \xi^{j_k}}.
	                               \label{Eq2.54}
	\end{equation}
	
	By the definition of the operator $J^k$,
	$$
	J^k\!f=\xi^{i_1}\dots\xi^{i_m}J^k\!f_{i_1\dots i_m}.
	$$
	Applying the operator $\frac{\partial^k}{\partial x^{j_1}\dots\partial x^{j_k}}$ to this equality, we obtain
	$$
	\frac{\partial^k(J^k\!f)}{\partial x^{j_1}\dots\partial x^{j_k}}=
	\xi^{i_1}\dots\xi^{i_m}\frac{\partial^k(J^k\!f_{i_1\dots i_m})}{\partial x^{j_1}\dots\partial x^{j_k}}.
	$$
	Together with \eqref{Eq2.54}, this gives
	\begin{equation}
	\frac{\partial^k(J^k\!f)}{\partial x^{j_1}\dots\partial x^{j_k}}=
	\xi^{i_1}\dots\xi^{i_m}\frac{\partial^k\psi_{i_1\dots i_m}}{\partial \xi^{j_1}\dots\partial \xi^{j_k}}.
	\label{Eq2.55}
	\end{equation}
	
	To compute the right-hand side of \eqref{Eq2.55}, we consider the expression for $\psi_{i_{1}\cdots i_{m}}$ from \eqref{Eq2.44}: 
	$$
	\psi_{i_1\dots i_m}=\frac{1}{m!}\sigma(i_1\dots i_m)\sum\limits_{\ell=0}^m(-1)^\ell{m\choose \ell}\frac{\partial^m \psi^\ell}
	{\partial x^{i_1}\dots\partial x^{i_\ell}\partial\xi^{i_{\ell+1}}\dots\partial\xi^{i_m}}
	$$
	and apply the operator $\frac{\partial^k}{\partial \xi^{j_1}\dots\partial \xi^{j_k}}$ to this equality
	$$
	\frac{\partial^k\psi_{i_1\dots i_m}}{\partial \xi^{j_1}\dots\partial \xi^{j_k}}=
	\frac{1}{m!}\sigma(i_1\dots i_m)\sum\limits_{\ell=0}^m(-1)^\ell{m\choose \ell}\frac{\partial^{m+k} \psi^\ell}
	{\partial x^{i_1}\dots\partial x^{i_\ell}\partial\xi^{i_{\ell+1}}\dots\partial\xi^{i_m}\partial \xi^{j_1}\dots\partial \xi^{j_k}}.
	$$
	We multiply this equality by $\xi^{i_1}\dots\xi^{i_m}$ and perform the summation over indices $(i_1,\dots,i_m)$. While doing this, we can omit the symmetrization $\sigma(i_1\dots i_m)$ since the factor $\xi^{i_1}\dots\xi^{i_m}$ is symmetric in these indices. In this way we obtain
	$$
	\xi^{i_1}\!\!\dots\xi^{i_m}\frac{\partial^k\psi_{i_1\dots i_m}}{\partial \xi^{j_1}\dots\partial \xi^{j_k}}=
	\frac{1}{m!}\sum\limits_{\ell=0}^m(-1)^\ell{m\choose \ell}\xi^{i_1}\!\!\dots\xi^{i_m}\frac{\partial^{m+k} \psi^\ell}
	{\partial x^{i_1}\dots\partial x^{i_\ell}\partial\xi^{i_{\ell+1}}\dots\partial\xi^{i_m}\partial \xi^{j_1}\dots\partial \xi^{j_k}}.
	$$
	This can be written in the form
	\begin{equation}
	\xi^{i_1}\!\!\dots\xi^{i_m}\frac{\partial^k\psi_{i_1\dots i_m}}{\partial \xi^{j_1}\dots\partial \xi^{j_k}}=
	\frac{1}{m!}\sum\limits_{\ell=0}^m(-1)^\ell{m\choose \ell}\l\xi,\partial_x\r^\ell\Big(\xi^{i_{\ell+1}}\dots\xi^{i_m}\frac{\partial^{m+k-\ell} \psi^\ell}
	{\partial\xi^{i_{\ell+1}}\!\!\dots\partial\xi^{i_m}\partial \xi^{j_1}\dots\partial \xi^{j_k}}\Big).
	\label{Eq2.56}
	\end{equation}
	
	By \eqref{Eq2.23}, the function $\frac{\partial^{k} \psi^\ell}{\partial \xi^{j_1}\dots\partial \xi^{j_k}}$ is positively homogeneous of degree $m-k-\ell-1$ in $\xi$.
	By Lemma \ref{L2.1},
	$$
	\begin{aligned}
	\xi^{i_{\ell+1}}\dots\xi^{i_m}&\frac{\partial^{m+k-\ell} \psi^\ell}
	{\partial\xi^{i_{\ell+1}}\dots\partial\xi^{i_m}\partial \xi^{j_1}\dots\partial \xi^{j_k}}=
	\xi^{i_{\ell+1}}\dots\xi^{i_m}\frac{\partial^{m-\ell}}{\partial\xi^{i_{\ell+1}}\dots\partial\xi^{i_m}}\Big(
	\frac{\partial^{k} \psi^\ell}{\partial \xi^{j_1}\dots\partial \xi^{j_k}}\Big)\\
	&=(m-k-\ell-1)(m-k-\ell-2)\dots(-k)\frac{\partial^{k} \psi^\ell}{\partial \xi^{j_1}\dots\partial \xi^{j_k}}.
	\end{aligned}
	$$
	If $m-k-\ell-1\ge0$, then the product $(m-k-\ell-1)(m-k-\ell-2)\dots(-k)$ contains the zero factor. Otherwise, if $\ell\ge m-k$,
	$$
	(m-k-\ell-1)(m-k-\ell-2)\dots(-k)=(-1)^{m-\ell}\frac{k!}{(k+\ell-m)!}.
	$$
	Thus,
	$$
	\xi^{i_{\ell+1}}\dots\xi^{i_m}\frac{\partial^{m+k-\ell} \psi^\ell}
	{\partial\xi^{i_{\ell+1}}\dots\partial\xi^{i_m}\partial \xi^{j_1}\dots\partial \xi^{j_k}}=
	\left\{
	\begin{aligned}
	&0\quad\mbox{if}\quad \ell\le m-k-1,\\
	&(-1)^{m-\ell}\frac{k!}{(k+\ell-m)!}\frac{\partial^{k} \psi^\ell}{\partial \xi^{j_1}\dots\partial \xi^{j_k}}
	\ \mbox{if}\ \ell\ge m-k.
	\end{aligned}\right.
	$$
	Substitute this value into \eqref{Eq2.56} to obtain
	$$
	\xi^{i_1}\!\!\dots\xi^{i_m}\frac{\partial^k\psi_{i_1\dots i_m}}{\partial \xi^{j_1}\dots\partial \xi^{j_k}}=
	\frac{(-1)^m}{m!}\sum\limits_{\ell=m-k}^m \frac{k!}{(k+\ell-m)!}{m\choose \ell}\l\xi,\partial_x\r^\ell\Big(
	\frac{\partial^{k} \psi^\ell}{\partial \xi^{j_1}\dots\partial \xi^{j_k}}\Big).
	$$
	Together with \eqref{Eq2.55}, this gives
	\begin{equation}
	\frac{\partial^k(J^k\!f)}{\partial x^{j_1}\dots\partial x^{j_k}}=
	\frac{(-1)^m}{m!}\sum\limits_{\ell=m-k}^m \frac{k!}{(k+\ell-m)!}{m\choose \ell}\l\xi,\partial_x\r^\ell\Big(
	\frac{\partial^{k} \psi^\ell}{\partial \xi^{j_1}\dots\partial \xi^{j_k}}\Big).
	                      \label{Eq2.57}
	\end{equation}
	
	By Lemma \ref{L2.2},
	\begin{equation}
	\l\xi,\partial_x\r^\ell\Big(\frac{\partial^{k} \psi^\ell}{\partial \xi^{j_1}\dots\partial \xi^{j_k}}\Big)
	=\sigma(j_1\dots j_k)\sum\limits_{p=0}^{\min(k,\ell)}(-1)^p{k\choose p}\frac{\ell!}{(\ell-p)!}\,
	\frac{\partial^{k}(\l\xi,\partial_x\r^{\ell-p} \psi^\ell)}{\partial x^{j_1}\dots\partial x^{j_p}\partial \xi^{j_{p+1}}\dots\partial \xi^{j_k}}.
	                       \label{Eq2.58}
	\end{equation}
	By \eqref{Eq2.45},
	$$
	\l\xi,\partial_x\r^{\ell-p} \psi^\ell=(-1)^{\ell-p}{\ell\choose p}(\ell-p)!\,\psi^p.
	$$
	Substituting this value into \eqref{Eq2.58}, we get,
	$$
	\l\xi,\partial_x\r^\ell\Big(\frac{\partial^{k} \psi^\ell}{\partial \xi^{j_1}\dots\partial \xi^{j_k}}\Big)
	=(-1)^\ell\sigma(j_1\dots j_k)\sum\limits_{p=0}^{\min(k,\ell)} \ell!{k\choose p}{\ell\choose p}\,
	\frac{\partial^{k}\psi^p}{\partial x^{j_1}\dots\partial x^{j_p}\partial \xi^{j_{p+1}}\dots\partial \xi^{j_k}}.
	$$
	Then substituting the latter expression into \eqref{Eq2.57}, we have
		$$
	\begin{aligned}
	\frac{\partial^k(J^k\!f)}{\partial x^{j_1}\dots\partial x^{j_k}}=
	\frac{(-1)^m}{m!}\sigma(j_1\dots j_k)\sum\limits_{\ell=m-k}^m& (-1)^\ell\frac{k!\ell!}{(k+\ell-m)!}{m\choose \ell}
	\sum\limits_{p=0}^{\min(k,\ell)}{k\choose p}{\ell\choose p} \times\\
	&\times\frac{\partial^{k}\psi^p}{\partial x^{j_1}\dots\partial x^{j_p}\partial \xi^{j_{p+1}}\dots\partial \xi^{j_k}}.
	\end{aligned}
	$$
	After changing the order of summations, this can be written as
	\begin{equation}
	\frac{\partial^k(J^k\!f)}{\partial x^{j_1}\dots\partial x^{j_k}}=
	\sigma(j_1\dots j_k)\sum\limits_{p=0}^k a(m,k,p)\,\frac{\partial^{k}\psi^p}{\partial x^{j_1}\dots\partial x^{j_p}\partial \xi^{j_{p+1}}\dots\partial \xi^{j_k}},
	                       \label{Eq2.59}
	\end{equation}
	where
		\begin{equation}
	a(m,k,p)=\frac{(-1)^m}{m!}{k\choose p}\sum\limits_{\ell=\max(m-k,p)}^m (-1)^\ell\frac{k!\ell!}{(k+\ell-m)!}{m\choose \ell}{\ell\choose p}.
	                     \label{Eq2.60}
	\end{equation}
	As we will show
	\begin{equation}
	a(m,k,p)=\left\{\begin{array}{l}0\quad\mbox{for}\quad p<k\le m,\\ 1\quad\mbox{for}\quad p=k\le m.\end{array}\right.
	                           \label{Eq2.61}
	\end{equation}
	Substituting this value into \eqref{Eq2.59}, we obtain
	$$
	\frac{\partial^k(J^k\!f)}{\partial x^{j_1}\dots\partial x^{j_k}}=
	\frac{\partial^{k}\psi^k}{\partial x^{j_1}\dots\partial x^{j_k}}.
	$$
	This coincides with \eqref{Eq2.53}.
	
	It remains to prove \eqref{Eq2.61}.
	As is seen from \eqref{Eq2.60}, $a(m,0,0)=1$; this agrees with \eqref{Eq2.61}. Therefore we assume $k>0$ in what follows.
	
	First of all we change the summation variable in \eqref{Eq2.60} as $\ell=r+m-k$
	$$
	a(m,k,p)=\frac{(-1)^k}{m!}{k\choose p}\sum\limits_{r=\max(0,k+p-m)}^k (-1)^r\,\frac{k!(r+m-k)!}{r!}{m\choose {k-r}}{{r+m-k}\choose p}.
	$$
	Simplifying this, we obtain
	\[
	a(m,k,p)=(-1)^{k}{k\choose p}c(m,k,p),
	\]
	where
	\[
	c(m,k,p)=\sum\limits_{r=0}^{k}(-1)^r {k\choose r}{r+m-k\choose p}.
	\]
	This expression for $c(k,m,p)$ is well-known; see \cite[Formula 47, Section 4.2.5]{PBM}. Namely,
	\[
	c(m,k,p)=
	\begin{cases}
	0\quad \mbox{if} \quad p<k\leq m,\\
	(-1)^k\quad \mbox{if} \quad p=k\leq m.
	\end{cases}
	\]
Substituting this value into the previous formula, we obtain \eqref{Eq2.61}.
\end{proof}

Statement \ref{S2.12} establishes that $I^{k}f=\vp^{k}$ for all $0\leq k\leq m$.  Indeed, comparing \eqref{Eq2.21} and \eqref{Eq2.52}, we see that
$$
(\varphi^0,\dots,\varphi^m)=(I^0\!f,\dots,I^m\!f).
$$
We have completed the proof of the sufficiency part of Theorem \ref{Th1.3}.

\section{Proof of Theorem \ref{Th1.4}}\label{S3}

A symmetric rank $m$ tensor field $f$ on the plane ${\R}^2=\{(x_1,x_2)\}$ is uniquely written as
$$
f=\sum\limits_{j=0}^m {m\choose j}f_{\underbrace{1\dots1}_{m-j}\underbrace{2\dots2}_j}(x_1,x_2)\,dx_1^{m-j}dx_2^j.
$$
Introducing the functions
$$
{\check f}_j(x_1,x_2)={m\choose j}f_{\underbrace{\small 1\dots1}_{m-j}\underbrace{2\dots2}_j}(x_1,x_2)\quad(j=0,1,\dots, m),
$$
we write this in the form
$$
f=\sum\limits_{j=0}^m{\check f}_j(x_1,x_2)\,dx_1^{m-j}dx_2^j.
$$
We refer to ${\check f}_j$ as {\it real coordinates} of the tensor field $f$ (although they are complex-valued functions in the general case).

We identify ${\R}^2$ with the complex plane $\C=\{z=x_1+ \textsl{i}\,x_2\}$ (hereafter $\textsl{i}$ is the imaginary unit). The covectors $dz=dx_1+\textsl{i}\,dx_2$ and $d\bar z=dx_1-\textsl{i}\,dx_2$ generate the algebra of symmetric tensor fields, i.e., every symmetric rank $m$ tensor field $f$ is uniquely represented in the form
$$
f=\sum\limits_{j=0}^m {\tilde f}_j(z)\,dz^{m-j}d\bar z^j.
$$
We refer to ${\tilde f}_j$ as {\it complex coordinates} of the tensor field $f$

Two bases $\{dx_1^{m-j}dx_2^j\mid0\le j\le m\}$ and $\{dz^{m-j}d\bar z^j\mid0\le j\le m\}$ of the space $S^m{\R}^2$ are related by
\begin{equation}
dz^{m-j}d\bar z^j=\sum\limits_{q=0}^m a^j_q\, dx_1^{m-q}dx_2^q,
                                         \label{Eq3.1}
\end{equation}
with some nondegenerate $(m+1)\times(m+1)$-matrix $A_m=(a^j_q)_{j,q=0}^m$. The components of a tensor field are transformed by
\begin{equation}
{\tilde f}_j(z)=\sum\limits_{q=0}^m b^q_j\,{\check f}_q(x_1,x_2)\quad(z=x_1+ \textsl{i}x_2),
                                         \label{Eq3.2}
\end{equation}
where $B_m=(b^q_j)_{j,q=0}^m$ is the inverse matrix of $A_m$.

We write a point $(x,\xi)\in{\R}^2\times({\R}^2\setminus\{0\})$ as $(z,\zeta)\in{\C}\times({\C}\setminus\{0\})$, where $z=x_1+\textsl{i}\,x_2$ and $\zeta=\xi_1+\textsl{i}\,\xi_2$.
The momentum ray transform
$$
J^k:{\mathcal S}({\R}^2;S^m{\R}^2)\rightarrow C^\infty\big({\C}\times({\C}\setminus\{0\})\big)
$$
is written as
\begin{equation}
(J^k\!f)(z,\zeta)=\int\limits_{-\infty}^\infty t^k\sum\limits_{j=0}^m
{\check f}_j(z+t\zeta)\,
\xi_1^{m-j}\xi_2^j\,dt\quad (\zeta=\xi_1+\textsl{i}\,\xi_2).
                                         \label{Eq3.3}
\end{equation}
This is written in terms of complex coordinates of $f$ as follows:
\begin{equation}
(J^k\!f)(z,\zeta)=\int\limits_{-\infty}^\infty t^k\sum\limits_{j=0}^m
{\tilde f}_j(z+t\zeta)\,\zeta^{m-j}\bar\zeta^j\,dt.
                                         \label{Eq3.4}
\end{equation}
Indeed, the integrands in \eqref{Eq3.3} and \eqref{Eq3.4} coincide as is seen from \eqref{Eq3.1}--\eqref{Eq3.2}.

Points of $T{\mathbb S}^1$ are uniquely written as $(\textsl{i}\,p\zeta,\zeta)$, where $\zeta\in{\C},|\zeta|=1$ and $p\in\R$. In complex variables, Theorem \ref{Th1.4} is as follows.

\begin{theorem} \label{Th4.1}
Let $m\ge0$. If an $(m+1)$-tuple
$$
(\varphi^0,\dots,\varphi^m)\in\underbrace{{\mathcal S}(T{\mathbb S}^{1})\times\dots\times{\mathcal S}(T{\mathbb S}^{1})}_{m+1}
$$
belongs to the range of the operator
$$
{\mathcal S}({\R}^2;S^m{\R}^2)\rightarrow\underbrace{{\mathcal S}(T{\mathbb S}^{1})\times\dots\times{\mathcal S}(T{\mathbb S}^{1})}_{m+1},\quad
f\mapsto(I^0\!f,\dots,I^m\!f),
$$
then the following conditions are satisfied.

{\rm (1)} $\varphi^k(\textsl{i}\,p\zeta,-\zeta)=(-1)^{m-k}\varphi^k(\textsl{i}\,p\zeta,\zeta)\ (0\le k\le m)$.

{\rm (2)} For every $r=0,1,2,\dots$ and for every $k=0,1,\dots,m$
\begin{equation}
\int\limits_{-\infty}^\infty p^r\,\varphi^k(\textsl{i}\,p\zeta,\zeta)\,dp=P^{rk}(\zeta)
\quad\mbox{for}\quad|\zeta|=1,
                                         \label{Eq3.5}
\end{equation}
where $P^{rk}(\zeta)$ are homogeneous polynomials of degree $r+k+m$ in $(\zeta,\bar\zeta)$.

{\rm (3)}  Polynomials $P^{rk}(\zeta)$ are not independent. They are described by the following construction. For
every pair $(\alpha,\beta)$ of non-negative integers there exists a symmetric $m$-tensor
$\mu^{\alpha\beta}=(\tilde\mu^{\alpha\beta}_j)\in S^m{\R}^2$ such that
\begin{equation}
P^{rk}(\zeta)=\frac{\textsl{i\,}^r}{2^{r+k}}\sum\limits_{j=0}^m\sum\limits_{\alpha=0}^r\sum\limits_{\beta=0}^k
(-1)^\alpha{r\choose\alpha}{k\choose\beta}\tilde\mu^{\alpha+\beta,r+k-\alpha-\beta}_j
\zeta^{m+r+k-j-\alpha-\beta}\bar\zeta^{j+\alpha+\beta}.
                                         \label{Eq3.6}
\end{equation}

Conversely, if functions
$\varphi^k\in{\mathcal S}(T{\mathbb S}^{1})\ (k=0,\dots m)$
satisfy conditions {\rm (1)}--{\rm (3)} with some tensors $\mu^{\alpha\beta}\in S^m{\R}^2$, then there exists a tensor field $f\in{\mathcal S}({\R}^2;S^m{\R}^2)$ such that $(\varphi^0,\dots,\varphi^m)=(I^0\!f,\dots,I^m\!f)$.
\end{theorem}

The equivalence of Theorems \ref{Th1.4} and \ref{Th4.1} is easily proved with the help of \eqref{Eq3.1}--\eqref{Eq3.2}; we omit the details.

\begin{proof}[Proof of the necessity in Theorem \ref{Th4.1}]
For a tensor field $f\in{\mathcal S}({\R}^2;S^m{\R}^2)$, we introduce the {\it complex integral momenta}
\begin{equation}
{\tilde\mu}_j^{\alpha\beta}={\tilde\mu}_j^{\alpha\beta}(f)=
\int\limits_{\C}z^\alpha\bar z^\beta\,{\tilde f}_j(z)\,d\sigma(z).
                                        \label{Eq3.7}
\end{equation}
Here
$
d\sigma(z)=dx_1\wedge dx_2=\frac{\textsl{i}}{2}dz\wedge d\bar z
$
is the area form. We write \eqref{Eq3.4} in the form
$$
\varphi^k(\textsl{i}\,p\zeta,\zeta)=(I^k\!f)(\textsl{i}\,p\zeta,\zeta)=
\int\limits_{-\infty}^\infty t^k\sum\limits_{j=0}^m {\tilde f}_j(\textsl{i}\,p\zeta+t\zeta)\,\zeta^{m-j}\bar\zeta^j\,dt\quad(|\zeta|=1).
$$
Multiply this equality by $p^r$ and integrate with respect to $p$
\begin{equation}
\int\limits_{-\infty}^\infty p^r\,\varphi^k(\textsl{i\,}p\zeta,\zeta)\,dp=
\int\limits_{-\infty}^\infty \int\limits_{-\infty}^\infty p^rt^k\sum\limits_{j=0}^m {\tilde f}_j(\textsl{i\,}p\zeta+t\zeta)\,\zeta^{m-j}\bar\zeta^j\,dtdp.
                                        \label{Eq3.8}
\end{equation}
For a fixed $\zeta$ satisfying $|\zeta|=1$, we change integration variables in \eqref{Eq3.8} as $z=\textsl{i\,}p\zeta+t\zeta$. Then
$$
t=\Re(z\bar\zeta)=\frac{1}{2}(z\bar\zeta+\bar z\zeta),\quad
p=\Im(z\bar\zeta)=\frac{\textsl{i}}{2}(\bar z\zeta-z\bar\zeta),\quad
dt\wedge dp=\frac{\textsl{i}}{2}dz\wedge d\bar z= d\sigma.
$$
After the change, formula \eqref{Eq3.8} becomes
$$
\int\limits_{-\infty}^\infty p^r\,\varphi^k(\textsl{i\,}p\zeta,\zeta)\,dp=\frac{{\textsl{i\,}}^r}{2^{r+k}}
\int\limits_{\C} (\bar z\zeta-z\bar\zeta)^r(z\bar\zeta+\bar z\zeta)^k\sum\limits_{j=0}^m {\tilde f}_j(z)\,\zeta^{m-j}\bar\zeta^j\,d\sigma(z).
$$
Expanding $(\bar z\zeta-z\bar\zeta)^r$ and $(z\bar\zeta+\bar z\zeta)^k$ by Newton's binomial formula, we obtain
$$
\begin{aligned}
\int\limits_{-\infty}^\infty& p^r\,\varphi^k(\textsl{i\,}p\zeta,\zeta)\,dp=\\
&=\frac{\textsl{i\,}^r}{2^{r+k}}
\sum\limits_{j=0}^m\sum\limits_{\alpha=0}^r\sum\limits_{\beta=0}^k(-1)^\alpha{r\choose\alpha}{k\choose\beta}
\zeta^{m+r+k-j-\alpha-\beta}\bar\zeta^{j+\alpha+\beta}
\int\limits_{\C} z^{\alpha+\beta}\bar z^{r+k-\alpha-\beta} {\tilde f}_j(z)\,d\sigma(z).
\end{aligned}
$$
On using \eqref{Eq3.7}, this becomes
$$
\int\limits_{-\infty}^\infty p^r\,\varphi^k(\textsl{i\,}p\zeta,\zeta)\,dp=
\frac{\textsl{i\,}^r}{2^{r+k}}
\sum\limits_{j=0}^m\sum\limits_{\alpha=0}^r\sum\limits_{\beta=0}^k(-1)^\alpha{r\choose\alpha}{k\choose\beta}{\tilde\mu}_j^{\alpha+\beta,r+k-\alpha-\beta}
\zeta^{m+r+k-j-\alpha-\beta}\bar\zeta^{j+\alpha+\beta}.
$$
We write the result in the form
$$
\int\limits_{-\infty}^\infty p^r\,\varphi^k(\textsl{i}p\zeta,\zeta)\,dp= P^{rk}(\zeta)\quad(|\zeta|=1),
$$
where the polynomials $P^{rk}(\zeta)$ are given in \eqref{Eq3.6}.
\end{proof}

\begin{proof}[Proof of the sufficiency in Theorem \ref{Th4.1}]
The proof is by induction on $m$. In the case of $m=0$, Theorem \ref{Th4.1} actually coincides with Theorem \ref{Th1.2} (also for $m=0$).

Let functions $\varphi^k\in{\mathcal S}(T{\mathbb S}^{1})\ (k=0,\dots, m)$
satisfy conditions {\rm (1)}--{\rm (3)} of Theorem \ref{Th4.1} with some tensors $\mu^{\alpha\beta}\in S^m{\R}^2$.
In particular, the function $\varphi^0\in{\mathcal S}(T{\mathbb S}^{1})$ satisfies $\varphi^0(\textsl{i\,}p\zeta,-\zeta)=(-1)^m\varphi^0(\textsl{i\,}p\zeta,\zeta)$ and
\begin{equation}
\int\limits_{-\infty}^\infty p^r\,\varphi^0(\textsl{i\,}p\zeta,\zeta)\,dp =P^{r0}(\zeta)\quad (|\zeta|=1;\ r=0,1,\dots),
                                         \label{Eq3.9}
\end{equation}
where $P^{r0}(\zeta)$ is a homogeneous polynomial of degree $r+m$ in $(\zeta,\bar\zeta)$.
These are just the hypotheses of Theorem \ref{Th1.2}. Applying the theorem, we can state the existence of a tensor field
$g\in{\mathcal S}({\R}^2;S^m{\R}^2)$ such that
\begin{equation}
\varphi^0 =I^0g.
                                         \label{Eq3.10}
\end{equation}
We emphasize that such a tensor field is not unique, it is determined by \eqref{Eq3.10} up to an arbitrary potential field (the definition of a {\it potential tensor field} is presented in \cite{KMSS}, as well as the the definition of the {\it inner derivative} $d$ which is used in the next paragraph). Let us fix some tensor field $g\in{\mathcal S}({\R}^2;S^m{\R}^2)$ satisfying \eqref{Eq3.10}.

Now we look for a tensor field $f\in{\mathcal S}({\R}^2;S^m{\R}^2)$ of the form
\begin{equation}
f =g+dv,
                                         \label{Eq3.11}
\end{equation}
where $v\in{\mathcal S}({\R}^2;S^{m-1}{\R}^2)$. We are looking for a tensor field  $v\in{\mathcal S}({\R}^2;S^{m-1}{\R}^2)$ such that  the tensor field $f$ defined by \eqref{Eq3.11} satisfies
\begin{equation}
I^kf =\varphi^k\quad(k=0,\dots,m).
                                         \label{Eq3.12}
\end{equation}

Recall the relation of the inner derivative $d$ to momentum ray transforms \cite{KMSS}:
$$
I^k(dv) =-kI^{k-1}v.
$$
Together with \eqref{Eq3.11}--\eqref{Eq3.12}, this gives
\begin{equation}
I^kv =-\frac{1}{k+1}(\varphi^{k+1}-I^{k+1}g)\quad(k=0,\dots,m-1).
                                         \label{Eq3.13}
\end{equation}
Introducing the functions
\begin{equation}
\chi^k=-\frac{1}{k+1}(\varphi^{k+1}-I^{k+1}g)\quad(k=0,\dots,m-1),
                                         \label{Eq3.14}
\end{equation}
we write \eqref{Eq3.13} in the form
\begin{equation}
I^kv =\chi^k\quad(k=0,\dots,m-1).
                                         \label{Eq3.15}
\end{equation}

Observe that formula \eqref{Eq3.14} can be taken as the definition of the functions $\chi^k\in{\mathcal S}(T{\mathbb S}^1)\ (k=0,\dots,m-1)$; we do not need to know the tensor field $f$ for the definition. We are going to prove that the functions $(\chi^0,\dots,\chi^{m-1})$ defined by \eqref{Eq3.14} satisfy all hypotheses of Theorem \ref{Th4.1} with $m$ replaced with $m-1$. Then by the induction hypothesis, there exists a tensor field $v\in{\mathcal S}({\R}^2;S^{m-1}{\R}^2)$ satisfying \eqref{Eq3.15}. Given $g$ and $v$, we define the tensor field $f$ by \eqref{Eq3.11}. Inverting the above elementary arguments, we obtain \eqref{Eq3.12}. This would finish the induction step.

By the necessity part of Theorem \ref{Th4.1},
\begin{equation}
\int\limits_{-\infty}^\infty p^r(I^kg)(\textsl{i\,}p\zeta,\zeta)\,dp =Q^{rk}(\zeta)\quad (|\zeta|=1;\ r=0,1,\dots),
                                         \label{Eq3.16}
\end{equation}
where $Q^{rk}(\zeta)$ are homogeneous polynomials of degree $r+k+m$ in $(\zeta,\bar\zeta)$. These polynomials are expressed through the integral momenta
of the field $g$
$$
\tilde\nu_j^{\alpha\beta}=\tilde\nu_j^{\alpha\beta}(g)=
\int\limits_{\C}z^\alpha\bar z^\beta\,\tilde g_j(z)\,d\sigma(z)\quad(j=0,\dots,m)
$$
by
\begin{equation}
Q^{rk}(\zeta)=\frac{\textsl{i\,}^r}{2^{r+k}}\sum\limits_{j=0}^m\sum\limits_{\alpha=0}^r\sum\limits_{\beta=0}^k
(-1)^\alpha{r\choose\alpha}{k\choose\beta}\tilde\nu^{\alpha+\beta,r+k-\alpha-\beta}_j
\zeta^{m+r+k-j-\alpha-\beta}\bar\zeta^{j+\alpha+\beta}.
                                         \label{Eq3.17}
\end{equation}

Equations \eqref{Eq3.5}, \eqref{Eq3.14} and \eqref{Eq3.16} imply
\begin{equation}
\int\limits_{-\infty}^\infty p^r\,\chi^k(\textsl{i}\,p\zeta,\zeta)\,dp= R^{rk}(\zeta)
\quad\mbox{for}\quad|\zeta|=1,\ k=0,\dots,m-1,
                                         \label{Eq3.18}
\end{equation}
where
$$
R^{rk}(\zeta)=-\frac{1}{k+1}\big(P^{r,k+1}(\zeta)-Q^{r,k+1}(\zeta)\big).
$$
According to this equality, $R^{rk}(\zeta)$ is a homogeneous polynomial of degree $r+m+k+1$ in $(\zeta,\bar\zeta)$. But we need to represent the left-hand side of \eqref{Eq3.18} by a homogeneous polynomial of degree $r+m+k-1$ for $|\zeta|=1$. In other words, we have to prove that the difference $P^{r,k+1}(\zeta)-Q^{r,k+1}(\zeta)$ is actually of the form
$$
P^{r,k+1}(\zeta)-Q^{r,k+1}(\zeta)=\zeta\bar\zeta\, S^{rk}(\zeta)
$$
with some homogeneous polynomial $S^{rk}(\zeta)$ of degree $r+m+k-1$. This is the most difficult part of the proof. To prove this fact, we have equations \eqref{Eq3.9}--\eqref{Eq3.10} only.

From \eqref{Eq3.9}--\eqref{Eq3.10}, we see that
$$
P^{r0}(\zeta)=\int\limits_{-\infty}^\infty p^r\,\varphi^0(\textsl{i\,}p\zeta,\zeta)\,dp =
\int\limits_{-\infty}^\infty p^r\,(I^0g)(\textsl{i\,}p\zeta,\zeta)\,dp.
$$
Together with \eqref{Eq3.16}, this gives
\begin{equation}
P^{r0}(\zeta)-Q^{r0}(\zeta)=0\quad (|\zeta|=1;\ r=0,1,\dots).
                                         \label{Eq3.19}
\end{equation}

By \eqref{Eq3.6} and \eqref{Eq3.17},
\begin{equation}
\begin{aligned}
&P^{rk}(\zeta)-Q^{rk}(\zeta)=\\
&=\frac{\textsl{i\,}^r}{2^{r+k}}\sum\limits_{j=0}^m\sum\limits_{\alpha=0}^r\sum\limits_{\beta=0}^k
(-1)^\alpha{r\choose\alpha}{k\choose\beta}(\tilde\mu^{\alpha+\beta,r+k-\alpha-\beta}_j-\tilde\nu^{\alpha+\beta,r+k-\alpha-\beta}_j)
\zeta^{m+r+k-j-\alpha-\beta}\bar\zeta^{j+\alpha+\beta}.
\end{aligned}
                                         \label{Eq3.20}
\end{equation}
From \eqref{Eq3.19} and \eqref{Eq3.20},
$$
\sum\limits_{j=0}^m\sum\limits_{\alpha=0}^r
(-1)^\alpha{r\choose\alpha}(\tilde\mu^{\alpha,r-\alpha}_j-\tilde\nu^{\alpha,r-\alpha}_j)
\zeta^{m+r-j-\alpha}\bar\zeta^{j+\alpha}=0.
$$
Change the summation variable in the inner sum as $s=j+\alpha$
$$
\sum\limits_{j=0}^m\sum\limits_{s=j}^{r+j}
(-1)^{s-j}{r\choose{s-j}}(\tilde\mu^{s-j,r+j-s}_j-\tilde\nu^{s-j,r+j-s}_j)
\zeta^{m+r-s}\bar\zeta^{s}=0.
$$
After changing the order of summations, this becomes
$$
\sum\limits_{s=0}^{r+m}(-1)^s\Big[\sum\limits_{j=\max(0,s-r)}^{\min(s,m)}
(-1)^j{r\choose{s-j}}(\tilde\mu^{s-j,r+j-s}_j-\tilde\nu^{s-j,r+j-s}_j)\Big]
\zeta^{m+r-s}\bar\zeta^{s}=0.
$$
This equality holds identically in $\zeta\in{\mathbb S}^1$. Since the left-hand side is a homogeneous polynomial, all its coefficients must be equal to zero. We have thus obtained
\begin{equation}
\sum\limits_{j=\max(0,s-r)}^{\min(s,m)}
(-1)^j{r\choose{s-j}}(\tilde\mu^{s-j,r+j-s}_j-\tilde\nu^{s-j,r+j-s}_j)=0\quad (r=0,1,\dots;\ 0\le s\le r+m).
                                         \label{Eq3.21}
\end{equation}
In particular, setting $s=0$ in \eqref{Eq3.21}, we have
\begin{equation}
\tilde\mu^{0r}_0-\tilde\nu^{0r}_0=0\quad (r=0,1,\dots).
                                         \label{Eq3.22}
\end{equation}
On the other hand, setting $s=m+k+1$ and $r=k+1$ in \eqref{Eq3.21}, we obtain
\begin{equation}
\tilde\mu^{k+1,0}_m-\tilde\nu^{k+1,0}_m=0\quad (k=0,1,\dots).
                                         \label{Eq3.23}
\end{equation}
In what follows, we will use \eqref{Eq3.22} and \eqref{Eq3.23} only.

Increase $k$ by 1 in \eqref{Eq3.20}
\begin{equation}
\begin{aligned}
&P^{r,k+1}(\zeta)-Q^{r,k+1}(\zeta)=\\
&=\frac{\textsl{i\,}^r}{2^{r\!+\!k\!+\!1}}\sum\limits_{j=0}^m\sum\limits_{\alpha=0}^r\sum\limits_{\beta=0}^{k+1}
(-1)^\alpha{r\choose\alpha}{{k\!+\!1}\choose\beta}(\tilde\mu^{\alpha\!+\!\beta,r\!+\!k\!-\!\alpha\!-\!\beta\!+\!1}_j
-\tilde\nu^{\alpha\!+\!\beta,r\!+k\!-\!\alpha\!-\!\beta\!+\!1}_j)
\zeta^{m\!+\!r\!+\!k\!-\!j\!-\!\alpha\!-\!\beta\!+\!1}\bar\zeta^{j\!+\!\alpha\!+\!\beta}.
\end{aligned}
                                         \label{Eq3.24}
\end{equation}

For two polynomials $P(\zeta)$ and $Q(\zeta)$, let us denote
$$
P\equiv Q\ (\mbox{mod}\,{\mathcal P}^{m+r+k-1})
$$
if the difference $P(\zeta)-Q(\zeta)$ can be represented by a homogeneous polynomial of degree $m+r+k-1$ on the unit circle $|\zeta|=1$.
Observe that
\begin{equation}
\zeta^a\bar\zeta^b\equiv0\ (\mbox{mod}\,{\mathcal P}^{m+r+k-1})\quad\mbox{if}\ a>0,\,b>0,\,a+b=m+r+k+1.
                                         \label{Eq3.25}
\end{equation}

We are going to prove that $P^{r,k+1}-Q^{r,k+1}\equiv0\ (\mbox{mod}\,{\mathcal P}^{m+r+k-1})$. Thus, we can delete all monomials on the right-hand side of \eqref{Eq3.24} which are of the form \eqref{Eq3.25}.

First of all, all summands on the right-hand side of \eqref{Eq3.24} which correspond to $0<\beta<k+1$ are of the form \eqref{Eq3.25}. Deleting such summands, we obtain
\begin{equation}
\begin{aligned}
\textsl{i\,}^{-r}2^{r\!+\!k\!+\!1}&(P^{r,k+1}-Q^{r,k+1})\equiv
\sum\limits_{j=0}^m\sum\limits_{\alpha=0}^r
(-1)^\alpha{r\choose\alpha}(\tilde\mu^{\alpha,r\!+\!k\!-\!\alpha\!+\!1}_j
-\tilde\nu^{\alpha,r\!+k\!-\!\alpha\!+\!1}_j)
\zeta^{m\!+\!r\!+\!k\!-\!j\!-\!\alpha\!+\!1}\bar\zeta^{j\!+\!\alpha}\\
&+\sum\limits_{j=0}^m\sum\limits_{\alpha=0}^r
(-1)^\alpha{r\choose\alpha}(\tilde\mu^{\alpha\!+\!k\!+\!1,r\!-\!\alpha}_j
-\tilde\nu^{\alpha\!+\!k\!+\!1,r\!-\!\alpha}_j)
\zeta^{m\!+\!r\!-\!j\!-\!\alpha}\bar\zeta^{j\!+\!\alpha\!+\!k\!+\!1}\ \ (\mbox{mod}\,{\mathcal P}^{m+r+k-1}).
\end{aligned}
                                         \label{Eq3.26}
\end{equation}
On the right-hand side of \eqref{Eq3.26}, the first line contains terms from \eqref{Eq3.24} which correspond to $\beta=0$, and the second line contains terms from \eqref{Eq3.24} which correspond to $\beta=k+1$.

Second, all summands on the right-hand side of \eqref{Eq3.26} which correspond to $0<\alpha<r$ are of the form \eqref{Eq3.25}. Deleting such summands, we obtain
\begin{equation}
\begin{aligned}
\textsl{i\,}^{-r}2^{r\!+\!k\!+\!1}(P^{r,k+1}-Q^{r,k+1})&\equiv
\sum\limits_{j=0}^m(\tilde\mu^{0,r\!+\!k\!+\!1}_j-\tilde\nu^{0,r\!+k\!+\!1}_j)\zeta^{m\!+\!r\!+\!k\!-\!j\!+\!1}\bar\zeta^{j}\\
&+(-1)^r\sum\limits_{j=0}^m(\tilde\mu^{r,k\!+\!1}_j-\tilde\nu^{r,k\!+\!1}_j)\zeta^{m\!+\!k\!-\!j\!+\!1}\bar\zeta^{j+r}\\
&+\sum\limits_{j=0}^m(\tilde\mu^{k\!+\!1,r}_j-\tilde\nu^{k\!+\!1,r}_j)\zeta^{m\!+\!r\!-\!j}\bar\zeta^{j\!+\!k\!+\!1}\\
&+(-1)^r\sum\limits_{j=0}^m(\tilde\mu^{r\!+\!k\!+\!1,0}_j-\tilde\nu^{r\!+\!k\!+\!1,0}_j)
\zeta^{m\!-\!j}\bar\zeta^{r\!+\!k\!+\!j\!+\!1}\ (\mbox{mod}\,{\mathcal P}^{m+r+k-1}).
\end{aligned}
                                         \label{Eq3.27}
\end{equation}

On the right-hand side of \eqref{Eq3.27}, all summands of the first and second sums which correspond to $j>0$ are of the form \eqref{Eq3.25}. All summands of the third and fourth sums which correspond to $j<m$ are of the form \eqref{Eq3.25}. Deleting such summands, we obtain
\begin{equation}
\begin{aligned}
\textsl{i\,}^{-r}2^{r\!+\!k\!+\!1}&(P^{r,k+1}-Q^{r,k+1})\equiv
(\tilde\mu^{0,r\!+\!k\!+\!1}_0-\tilde\nu^{0,r\!+k\!+\!1}_0)\zeta^{m\!+\!r\!+\!k\!+\!1}
+(-1)^r(\mu^{r,k\!+\!1}_0-\nu^{r,k\!+\!1}_0)\zeta^{m\!+\!k\!+\!1}\bar\zeta^{r}\\
&+(\tilde\mu^{k\!+\!1,r}_m-\tilde\nu^{k\!+\!1,r}_m)\zeta^{r}\bar\zeta^{m\!+\!k\!+\!1}
+(-1)^r(\tilde\mu^{r\!+\!k\!+\!1,0}_m-\tilde\nu^{r\!+\!k\!+\!1,0}_m)
\bar\zeta^{m\!+\!r\!+\!k\!+\!1}\ (\mbox{mod}\,{\mathcal P}^{m+r+k-1}).
\end{aligned}
                                         \label{Eq3.28}
\end{equation}

In the case of $r=0$, \eqref{Eq3.28} looks as follows:
\begin{equation}
2^{k}(P^{0,k+1}-Q^{0,k+1})\equiv
(\tilde\mu^{0,k\!+\!1}_0-\tilde\nu^{0,k\!+\!1}_0)\zeta^{m\!+\!k\!+\!1}
+(\tilde\mu^{k\!+\!1,0}_m-\tilde\nu^{k\!+\!1,0}_m)
\bar\zeta^{m\!+\!k\!+\!1}\ (\mbox{mod}\,{\mathcal P}^{m+k-1}).
                                         \label{Eq3.29}
\end{equation}
In the case of $r>0$, we can delete the second and third terms on the right-hand side of \eqref{Eq3.28} since they are of the form \eqref{Eq3.25}. We thus obtain
\begin{equation}
\left.
\begin{aligned}
\textsl{i\,}^{-r}2^{r\!+\!k\!+\!1}&(P^{r,k+1}-Q^{r,k+1})\equiv
(\tilde\mu^{0,r\!+\!k\!+\!1}_0-\tilde\nu^{0,r\!+k\!+\!1}_0)\zeta^{m\!+\!r\!+\!k\!+\!1}\\
&+(-1)^r(\tilde\mu^{r\!+\!k\!+\!1,0}_m-\tilde\nu^{r\!+\!k\!+\!1,0}_m)
\bar\zeta^{m\!+\!r\!+\!k\!+\!1}\ (\mbox{mod}\,{\mathcal P}^{m+r+k-1})
\end{aligned}
\right\}\quad (r>0).
                                         \label{Eq3.30}
\end{equation}

By \eqref{Eq3.22}--\eqref{Eq3.23}, right-hand sides of \eqref{Eq3.29} and \eqref{Eq3.30} are equal to zero. Thus
$$
P^{r,k+1}-Q^{r,k+1}\equiv0\ (\mbox{mod}\,{\mathcal P}^{m+r+k-1})
\quad (r=0,1,\dots).
$$
This finishes the induction step.
\end{proof}


\begin{thebibliography}{1}

\bibitem{D}
A.S. Denisiuk.
\newblock On range condition of the tensor x-ray transforms in
  $\mathbb{R}^{n}$.
\newblock {\em Inverse Problems and Imaging}.
\newblock To appear.

\bibitem{G}
I.M. Gel'fand, M.I. Graev, and N.Ya. Vilenkin.
\newblock {\em Generalized Functions, Volume 5: Integral Geometry and
  Representation Theory}.
\newblock Inverse and Ill-posed Problems Series. AMS Chelsea Publishing, 1966.

\bibitem{Hb}
S. Helgason.
\newblock {\em The {R}adon transform}, volume~5 of {\em Progress in
  Mathematics}.
\newblock Birkh\"auser Boston, Inc., Boston, MA, second edition, 1999.

\bibitem{J}
F. John.
\newblock The ultrahyperbolic differential equation with four independent
  variables.
\newblock {\em Duke Math. J.}, (2):300--322, 1938.

\bibitem{KMSS}
V.P. Krishnan, R. Manna, S.K. Sahoo, and V.A. Sharafutdinov.
\newblock Momentum ray transforms.
\newblock {\em Inverse Problems and Imaging}, (3):679--701, 2019.

\bibitem{NSV}
N.S. Nadirashvili, V.A.~Sharafutdinov and S.G. Vl\u{a}du\c{t}.
\newblock The John equation for tensor tomography in three-dimensions.
\newblock {\em Inverse Problems}, (10), 2016.
\newblock 105013, 15 pp.

\bibitem{P}
E.Yu. Pantjukhina.
\newblock Description of the range of the {X}-ray transform in two-dimensional
  case.
\newblock {\em {In: Methods of Solutions of Inverse Problems}}, pages 80--89
  (in Russian).

\bibitem{PBM}
A.P. Prudnikov, Yu.A. Brychkov, and O.I.   Marichev.
\newblock {\em Integrals and Series: Volume 1: Elementary Functions}.
\newblock {Gordon and Breach Science Publishers}, 1986.

\bibitem{Sh}
V.A. Sharafutdinov.
\newblock {\em Integral geometry of tensor fields}.
\newblock Inverse and Ill-posed Problems Series. VSP, Utrecht, 1994.

\end{thebibliography}
\end{document}